\documentclass[a4paper,10pt,reqno]{article}
\usepackage{amsmath}
\usepackage{amssymb}
\usepackage{amsthm}
\usepackage{xcolor}
\usepackage[colorlinks=true, bookmarks=true, pdfstartview=FitH, linktocpage=true, linkcolor = magenta, citecolor = blue]{hyperref}
\usepackage{tikz-cd}
\usepackage{multicol}
\usepackage{indentfirst}

\hypersetup{
	colorlinks = true,
	linkcolor = magenta,
	citecolor = blue,
}
\newif\ifshowrevisions
\showrevisionsfalse
\newcommand{\revisioncolor}{\ifshowrevisions
	\color{red}\hypersetup{linkcolor=red,citecolor=red,urlcolor=red}\fi}
\DeclareRobustCommand{\revised}[1]{{\revisioncolor #1}}
\newenvironment{revisedpassage}{\revisioncolor}{}
\makeatletter
\newcommand{\marknewtag}{\ifshowrevisions
	\let\revision@tagform\tagform@
	\def\tagform@##1{\textcolor{red}{\revision@tagform{##1}}}\fi}
\makeatother
\newtheorem{thm}{Theorem}
\newtheorem{lem}{Lemma} 

\newtheorem{cor}{Corollary}

\newtheorem*{prbs}{Problem $\mathcal{P}(\Omega^\varepsilon)$}
\newtheorem*{prbse}{Problem $\mathcal{P}(\varepsilon;\Omega)$}
\newtheorem*{prbsm}{Problem $\mathcal{P}_M(\omega)$}

\begin{document}
		
		\title{Derivation of a linearly elastic elliptic membrane shell theory in the framework of magnetoelasticity}
		\author{Trung Hieu Giang
			\footnote{Addresses of Trung Hieu Giang: Department of Mathematical Analysis, Faculty of Mathematics and Physics, Charles University, Praha, Czech Republic.  Institute of Mathematics, Vietnam Academy of Science and Technology, 18 Hoang Quoc Viet, Hanoi, Vietnam. \textit{Email address: trung-hieu.giang@matfyz.cuni.cz}
			} 
		}
		\date{\today}
		\maketitle
		
		\begin{abstract}
			Starting from a three-dimensional model of linearized magnetoelasticity, we investigate the asymptotic behavior of a thin shell as its thickness approaches zero. Our analysis focuses on a family of linearly elastic elliptic membrane shells. Using $\Gamma$-convergence, we rigorously derive a two-dimensional linear magnetoelastic shell model and establish its relationship to the original three-dimensional model in terms of the convergence of minimizers.
		\end{abstract}
		
		\
		
		
		\
		
		
		
		\section{Introduction}
		\label{sec1}
		
		Magnetoelasticity describes how solids behave mechanically under the influence of magnetic fields. The magnetoelastic coupling is based on the presence of small magnetic domains within the material \cite{HS98}. On the one hand, these domains, given an external magnetic field, tend to orient themselves to produce a magnetically induced deformation of the body. On the other hand, the mechanical effect alters the orientation of these domains, thereby changing the material's magnetic response. For a deeper foundation of magnetoelasticity, we refer the reader to \cite{HS98, Bro66, DO14}.
		
		The mathematical study of magnetoelasticity has become a very active field in the past few decades, and thus, we will only provide a brief review here. In the case of \textit{nonlinear models}, the mathematical formulation in this framework often involves both the energy terms defined on the original stress-free configuration and the deformed state. Thus, the model has a \textit{Lagrangian-Eulerian} structure, and from a mathematical point of view, the analysis of such structures is very challenging. A well-known difficulty is that the magnetization is defined on the unknown deformed configuration. Therefore, suitable injectivity conditions on the deformation are needed to prevent interpenetration of matter and to justify the corresponding changes of variables. A classical way to enforce almost everywhere injectivity is the Ciarlet--Ne\v{c}as condition (see \cite{Cia88,CN87}). For this reason, in recent years, magnetoelasticity has received much attention from the mathematical community, focusing on aspects such as existence theory, dimensional reduction, etc. In the context of existence theory for magnetoelasticity, we refer to Kru\v z\'ik et al. \cite{KSZ15} and Bresciani et al. \cite{BDK23}. In the context of dimensional reduction, Bresciani \cite{Bre21} derived a linearized von K\'arm\'an model from nonlinear three-dimensional incompressible magnetoelastic plates in the static setting by means of $\Gamma$-convergence. Subsequently, Bresciani and Kru\v z\'ik \cite{BK23} also obtained linearized von K\'arm\'an models for compressible magnetoelastic plates in both static and quasistatic settings. Davoli et al. \cite{DKPS21} derived a nonlinear two-dimensional model for flat magnetoelastic thin films also using the $\Gamma$-convergence. In the case of \textit{linear models}, Kru\v z\'ik et al. \cite{KSZa15} derived a two-dimensional quasistatic plate model starting from three-dimensional linear magnetoelasticity. Here, it should be noted that the use of such a linear elastic parent model can itself be justified through a small strain limit. In classical elasticity, Dal Maso, Negri, and Percivale \cite{DMNP02} rigorously derived linearized elasticity from finite elasticity by means of $\Gamma$-convergence. In the magnetoelastic setting, Almi et al. \cite{AKM25} subsequently established the corresponding linearization result for a nonlinear magnetoelastic model. The resulting linear model is the starting point of our analysis. 
		
		To the best of our knowledge, no linear two-dimensional shell theory has been derived from the three-dimensional model of Almi et al. mentioned above. The derivation of two-dimensional shell theories in linear elasticity has been extensively studied. In this direction, we would like to mention \cite{Cia00} for the results regarding purely linear elastic shell models. We would also like to mention, for example, \cite{CMP18, CMP19, Pie22} for models subjected to confinement conditions; \cite{CA18, CA19} for viscoelasticity; and \cite{Pie21} for thermoelasticity. Inspired by the aforementioned results, the aim of our contribution is to establish a relationship between three-dimensional and two-dimensional models of magnetoelasticity for thin shells.
		
		In particular, we start from the following model given by Almi, Kru\v z\'ik, and Molchanova \cite{AKM25} (which is also discussed in \cite{DJ02})
		\begin{align}
			\label{sec1-eqn1}
			\mathcal{G}(\boldsymbol u,\boldsymbol m)
			:={}&\frac12\int_\Omega \mathbb C\big(\boldsymbol\epsilon(\boldsymbol u)+\boldsymbol e(\boldsymbol m)\big):
			\big(\boldsymbol\epsilon(\boldsymbol u)+\boldsymbol e(\boldsymbol m)\big)\,\mathrm{d} x \nonumber\\
			&+\frac12\int_\Omega|\nabla\boldsymbol m|^2\,\mathrm{d} x
			+\frac{\mu_0}{2}\int_{\mathbb R^3}|\nabla v_{\boldsymbol m}|^2\,\mathrm{d} x \nonumber\\
			& - \int_\Omega \boldsymbol{f} \cdot \boldsymbol{u} \,\mathrm{d} x - \int_\Omega \boldsymbol{h} \cdot \boldsymbol{m} \,\mathrm{d} x
		\end{align}
		for each $(\boldsymbol{u}, \boldsymbol{m}) \in (H^1(\Omega))^3 \times (H^1(\Omega))^3$, where $\Omega \subset \mathbb{R}^3$ is a domain, $\mathbb{C}$ is a positive definite tensor,
		\begin{equation}
			\label{sec1-eqn2}
			\boldsymbol\epsilon(\boldsymbol u)=\frac12(\nabla\boldsymbol u+\nabla\boldsymbol u^T),
			\qquad
			\boldsymbol e(\boldsymbol m)=-\boldsymbol m\otimes\boldsymbol m+\frac13\boldsymbol I,
			\qquad |\boldsymbol m|=1,
		\end{equation}
		$\boldsymbol{I}$ is the $3\times 3$ identity matrix, $\boldsymbol{f} \in L^2(\Omega;\mathbb{R}^3)$ is a body force, $\boldsymbol{h} \in L^2(\mathbb{R}^3;\mathbb{R}^3)$ represents an external magnetic field, and the magnetostatic potential satisfies
		\begin{equation}
			\label{sec1-eqn3}
			\operatorname{div}\big(-\mu_0\nabla v_{\boldsymbol m}+\chi_\Omega\boldsymbol m\big)=0
			\quad\text{in }\mathcal D'(\mathbb R^3).
		\end{equation}
		
		In \eqref{sec1-eqn1}, the first three terms represent the elastic, exchange, and magnetostatic energies, respectively. The positive constant $\mu_0$ is the magnetic permeability of vacuum, and $v_{\boldsymbol{m}}$ is the magnetostatic potential generated by $\boldsymbol{m}$, which is a solution to the Maxwell equation \eqref{sec1-eqn3} with $\chi_\Omega$ being the characteristic function. We would like to emphasize that the model is linear with respect to the displacement, but remains nonlinear in $\boldsymbol{m}$.
		
		In this paper, we focus on the case where $\Omega$ is a thin elliptic shell (see Section \ref{sec2} for the definition). We will perform a rigorous asymptotic analysis as the thickness approaches zero for the model given by \eqref{sec1-eqn1}--\eqref{sec1-eqn3} to identify a linearly two-dimensional elliptic membrane shell theory. Since it is more convenient to describe the solutions to the above problem as minimizers of the functional \eqref{sec1-eqn1}, we exploit the method of $\Gamma$-convergence, which has become very common in dimensional reduction problems in mathematical elasticity. Our result, as far as we are concerned, is the first situation where a two-dimensional linearized magnetoelastic model is derived rigorously for thin curved shells.
		
		The organization of this paper is as follows. In Section \ref{sec2} we present some background and notation. In Section \ref{sec3} we recall the formulation and the properties of a three-dimensional problem for general linearly magnetoelastic shells. In Section \ref{sec4}, when a linearly elliptic membrane shell is taken into account, we rescale the problem presented in the previous section so that it is posed in a domain independent of the thickness parameter. Finally, in Section \ref{sec5}, the $\Gamma$-convergence is performed and we obtain the desired two-dimensional shell model.
		
		\section{Geometrical preliminaries}
		\label{sec2}
		
		For details about the classical notions of differential geometry used in this paper, we refer the reader to \cite{Cia00,Cia05}.
		
		Throughout this paper, the Greek indices (except $\varepsilon$) and exponents range in the set $\{1,2\}$, Latin indices and exponents range in the set $\{1,2,3\}$ (unless they are used for indexing sequences or when otherwise indicated). The summation convention with respect to repeated indices and exponents is systematically used in conjunction with these two rules. 
		
		The notation $\mathbb{E}^3$ designates the three-dimensional Euclidean space; the Euclidean inner product and the vector product of $\boldsymbol{u}, \boldsymbol{v} \in \mathbb{E}^3$ are denoted $\boldsymbol{u} \cdot \boldsymbol{v}$ and $\boldsymbol{u} \revised{\wedge} \boldsymbol{v}$; the Euclidean norm of $\boldsymbol{u} \in \mathbb{E}^3$ is denoted $\left|\boldsymbol{u} \right|$. The notation $\delta^j_i$ designates the Kronecker symbol.
		
		Strong and weak convergences in any normed vector space are respectively denoted $\to \textrm{ and }\rightharpoonup.$
		
		Given an open subset $\Omega$ of $\mathbb{R}^n$, notations such as $L^2(\Omega)$, $H^1(\Omega)$, or $H^1_0 (\Omega)$, designate the usual Lebesgue and Sobolev spaces, and the notation $\mathcal{D} (\Omega)$ designates the space of all functions that are infinitely differentiable over $\Omega$ and have compact supports in $\Omega$. The notation $\left\| \cdot \right\|_X$ designates the norm in a normed vector space $X$. The spaces of vector-valued functions are denoted with boldface letters.
		
		A \textit{domain in} $\mathbb{R}^n$ is a bounded and connected open subset $\Omega$ of $\mathbb{R}^n$, whose boundary $\partial \Omega$ is Lipschitz-continuous in the sense of Adams \& Fournier \cite{AdF09}.
		
		Let $\omega$ be a domain in $\mathbb{R}^2$, let $y = (y_\alpha)$ denote a generic point in $\omega$, and let $\partial_\alpha := \partial / \partial y_\alpha$ and $\partial_{\alpha \beta} := \partial^2/\partial y_\alpha \partial y_\beta$. A mapping $\boldsymbol{\theta} \in \mathcal{C}^2(\overline{\omega}; \mathbb{E}^3)$ is an \emph{immersion} if the two vectors
		\[
		\boldsymbol{a}_\alpha (y) := \partial_\alpha \boldsymbol{\theta} (y)
		\]
		are linearly independent at each point $y \in \overline{\omega}$. Then the image $\boldsymbol{\theta} (\overline{\omega})$ of the set $\overline{\omega}$ under the mapping $\boldsymbol{\theta}$ is a \emph{surface in} $\mathbb{E}^3$, equipped with $y_1, y_2$ as its \emph{curvilinear coordinates}. Given any point $y\in \overline{\omega}$, the vectors $\boldsymbol{a}_\alpha (y)$ span the \emph{tangent plane} to the surface $\boldsymbol{\theta} (\overline{\omega})$ at the point $\boldsymbol{\theta} (y)$, the unit vector
		\[
		\boldsymbol{a}_3 (y) := \frac{\boldsymbol{a}_1(y) \wedge \boldsymbol{a}_2 (y)}{|\boldsymbol{a}_1(y) \wedge \boldsymbol{a}_2 (y)|}
		\]
		is normal to $\boldsymbol{\theta} (\overline{\omega})$ at $\boldsymbol{\theta} (y)$, the three vectors $\boldsymbol{a}_i(y)$ form the \emph{covariant} basis at $\boldsymbol{\theta} (y)$, and the three vectors $\boldsymbol{a}^j(y)$ defined by the relations
		\[
		\boldsymbol{a}^j(y) \cdot \boldsymbol{a}_i (y) = \delta^j_i
		\]
		form the \textit{contravariant} basis at $\boldsymbol{\theta} (y)$; note that the vectors $\boldsymbol{a}^\beta (y)$ also span the tangent plane to $\boldsymbol{\theta} (\omega)$ at $\boldsymbol{\theta} (y)$ and that $\boldsymbol{a}^3(y) = \boldsymbol{a}_3 (y)$.
		
		The \textit{first fundamental form} of the surface $\boldsymbol{\theta} (\overline{\omega})$, given as the metric tensor, in covariant or contravariant components, is defined, respectively, by
		\[
		a_{\alpha \beta} := \boldsymbol{a}_\alpha \cdot \boldsymbol{a}_\beta = a_{\beta \alpha} \in \mathcal{C}^1 (\overline{\omega}),
		\]
		and
		\[
		a^{\alpha \beta}:= \boldsymbol{a}^\alpha \cdot \boldsymbol{a}^\beta = a^{\beta \alpha}\in \mathcal{C}^1(\overline{\omega}).
		\]
		Note that the symmetric matrix field $(a^{\alpha \beta})$ is then the inverse of the matrix field $(a_{\alpha \beta})$, that $\boldsymbol{a}^\beta = a^{\alpha \beta}\boldsymbol{a}_\alpha$ and $\boldsymbol{a}_\alpha = a_{\alpha \beta} \boldsymbol{a}^\beta$, and that the \textit{area element} along $\boldsymbol{\theta} (\overline{\omega})$ is given at each point $\boldsymbol{\theta} (y), \, y \in \overline{\omega}$, by $\sqrt{a(y)}\,\mathrm{d} y$, where
		\[
		a := \det (a_{\alpha \beta}) \in \mathcal{C}^1 (\overline{\omega}).  
		\]
		
		Given an immersion $\boldsymbol{\theta} \in \mathcal{C}^2(\overline{\omega}; \mathbb{E}^3)$, the \textit{second fundamental form} of the surface $\boldsymbol{\theta} (\overline{\omega})$ is defined by means of its \textit{covariant components}
		\[
		b_{\alpha \beta}:= \partial_\alpha \boldsymbol{a}_\beta \cdot \boldsymbol{a}_3 = - \boldsymbol{a}_\alpha \cdot \partial_\beta \boldsymbol{a}_3 = b_{\beta \alpha} \in \mathcal{C}^0(\overline{\omega}),
		\]
		or by means of its \textit{mixed components}
		\[
		b^\beta_\alpha := a^{\beta \sigma} b_{\alpha \sigma} \in \mathcal{C}^0(\overline{\omega}),
		\]
		and the \textit{Christoffel symbols} associated with the immersion $\boldsymbol{\theta}$ are defined by
		\[
		\Gamma^\sigma_{\alpha \beta}:= \partial_\alpha \boldsymbol{a}_\beta \cdot \boldsymbol{a}^\sigma = \Gamma^\sigma_{\beta \alpha} \in \mathcal{C}^0 (\overline{\omega}).
		\]
		
		The \textit{Gaussian curvature} at each point $\boldsymbol{\theta} (y) , \, y \in \overline{\omega}$, of the surface $\theta (\overline{\omega})$ is defined by
		\[
		K (y) := \frac{\det (b_{\alpha \beta} (y))}{\det (a_{\alpha \beta} (y))} = \det \left( b^\beta_\alpha (y)\right)
		\]
		(the denominator in the above relation does not vanish since $\boldsymbol{\theta}$ is assumed to be an immersion). Recall that the Gaussian curvature $K (y)$ at the point $\boldsymbol{\theta} (y)$ is also equal to the product of the two principal curvatures at this point.
		
		Given an immersion
		$\boldsymbol{\theta} \in \mathcal{C}^2 (\overline{\omega}; \mathbb{E}^3)$ and a
		vector field $\boldsymbol{\eta} = (\eta_i) \in \mathcal{C}^1(\overline{\omega};
		\mathbb{R}^3)$, the vector field
		\[
		\tilde{\boldsymbol{\eta}} := \eta_i \boldsymbol{a}^i\in \mathcal{C}^1(\overline{\omega};
		\mathbb{E}^3)
		\]
		can be viewed as a \textit{displacement field of the surface} $\boldsymbol{\theta} (\overline{\omega})$, thus defined by means of its \textit{covariant components} $\eta_i$ over the vectors $\boldsymbol{a}^i$ of the contravariant bases along the surface. If the norms $\left\| \eta_i \right\|_{\mathcal{C}^1(\overline{\omega})}$ are small enough, the mapping $(\boldsymbol{\theta} + \eta_i \boldsymbol{a}^i) \in \mathcal{C}^1(\overline{\omega}; \mathbb{E}^3)$ is also an immersion, so that the set $(\boldsymbol{\theta} + \eta_i \boldsymbol{a}^i) (\overline{\omega})$ is also a surface in $\mathbb{E}^3$, equipped with the same curvilinear coordinates as those of the surface $\boldsymbol{\theta} (\overline{\omega})$, called the \textit{deformed surface} corresponding to the displacement field $\tilde{\boldsymbol{\eta}} = \eta_i \boldsymbol{a}^i$. One can then define the first fundamental form of the deformed surface by means of its covariant components
		\begin{align*}
			a_{\alpha \beta} (\boldsymbol{\eta}) :=& (\boldsymbol{a}_\alpha + \partial_\alpha \tilde{\boldsymbol{\eta}}) \cdot (\boldsymbol{a}_\beta + \partial_\beta \tilde{\boldsymbol{\eta}}) \\
			=& a_{\alpha \beta} + \boldsymbol{a}_\alpha \cdot \partial_\beta \tilde{\boldsymbol{\eta}} + \partial_\alpha \tilde{\boldsymbol{\eta}} \cdot \boldsymbol{a}_\beta + \partial_\alpha \tilde{\boldsymbol{\eta}} \cdot \partial_\beta \tilde{\boldsymbol{\eta}}.
		\end{align*}
		
		The \textit{linear part with respect to} $\tilde{\boldsymbol{\eta}}$ in the difference $\dfrac12 (a_{\alpha \beta}(\boldsymbol{\eta}) - a_{\alpha \beta})$ is called the \emph{linearized change of metric}, or \textit{strain}, \emph{tensor} associated with the displacement field $\eta_i \boldsymbol{a}^i$, the covariant components of which are thus defined by
		\[
		\gamma_{\alpha \beta}(\boldsymbol{\eta}) := \dfrac12 \left( \boldsymbol{a}_\alpha \cdot \partial_\beta \tilde{\boldsymbol{\eta}} + \partial_\alpha \tilde{\boldsymbol{\eta}} \cdot \boldsymbol{a}_\beta \right) = \frac12 (\partial_\beta \eta_\alpha + \partial_\alpha \eta_\beta ) - \Gamma^\sigma_{\alpha \beta} \eta_\sigma - b_{\alpha \beta} \eta_3 = \gamma_{\beta \alpha} (\boldsymbol{\eta}).
		\]
		
		Let $\omega$ be a domain in $\mathbb{R}^2$. Then a surface $\boldsymbol{\theta} (\overline{\omega})$ defined by an immersion $\boldsymbol{\theta} \in \mathcal{C}^2(\overline{\omega};\mathbb{E}^3)$ is said to be \textit{elliptic} if its Gaussian curvature is everywhere positive in $\overline{\omega}$, or equivalently, if there exists a constant $K_0$ such that
		\[
		0 < K_0 \leq K (y ) \text{ for all } y \in \overline{\omega}.
		\]
		
		Thanks to the ellipticity, we have the following Korn inequality regarding the elliptic surfaces.
		\begin{thm}[Korn inequality on an elliptic surface \cite{CL96, CP96}]\label{sec2-thm1}
			Let $\omega$ be a domain in $\mathbb{R}^2$ and let an immersion $\boldsymbol{\theta} \in \mathcal{C}^{\revised{3}}(\overline{\omega};\mathbb{E}^3)$ be given such that the surface $\boldsymbol{\theta}(\overline{\omega})$ is elliptic. Define the space
			$$
			\boldsymbol{V}_M (\omega) := H^1_0(\omega) \times H^1_0(\omega) \times L^2(\omega).
			$$
			Then there exists a positive constant $c_0$ depending on $\omega$ and $\boldsymbol{\theta}$ such that
			\begin{equation*}
				\left\{ \sum_{\alpha} \|\eta_{\alpha}\|_{H^1(\omega)}^2 + \|\eta_3\|_{L^2(\omega)}^2 \right\}^{1/2} \leq c_0 \left\{ \sum_{\alpha, \beta} \|\gamma_{\alpha \beta}(\boldsymbol{\eta})\|_{L^2(\omega)}^2 \right\}^{1/2}
			\end{equation*}
			for all $\boldsymbol{\eta} =(\eta_i) \in \boldsymbol{V}_M (\omega)$.
		\end{thm}

		\section{The 3D problem in linear magnetoelasticity}
		\label{sec3}
		
		Let $\omega \subset \mathbb{R}^2$ be a domain with $\gamma:=\partial\omega$, and let $\gamma_0$ be a nonempty relatively open subset of $\gamma$. For each $\varepsilon>0$, define
		\[
		\Omega^\varepsilon:=\omega\times(-\varepsilon,\varepsilon),
		\qquad
		\Gamma_0^\varepsilon:=\gamma_0\times[-\varepsilon,\varepsilon].
		\]
		We let $x^\varepsilon=(x_i^\varepsilon)=(y,x_3^\varepsilon)$ designate a
		generic point of $\overline{\Omega^\varepsilon}$, and let
		$\partial_i^\varepsilon:=\partial/\partial x_i^\varepsilon$. Hence we also have
		$x_\alpha^\varepsilon=y_\alpha$ and $\partial_\alpha^\varepsilon=\partial_\alpha$.
		
		Given an \revised{injective} immersion
		$\boldsymbol\theta\in\mathcal C^3(\overline\omega;\mathbb E^3)$, consider
		a shell with middle surface $\boldsymbol\theta(\overline\omega)$ and
		constant thickness $2\varepsilon$. Its reference configuration is the set
		\[
		\widehat\Omega^\varepsilon
		:=\boldsymbol\Theta(\overline{\Omega^\varepsilon}),
		\]
		where
		\begin{equation*}
			\hat{x}^\varepsilon = \boldsymbol\Theta(x^\varepsilon)
			:=\boldsymbol\theta(y)+x_3^\varepsilon\boldsymbol a_3(y),
			\qquad x^\varepsilon=(y,x_3^\varepsilon)\in\overline{\Omega^\varepsilon}.
		\end{equation*}
		The half-thickness $\varepsilon$ is assumed to be small enough such that
		$\boldsymbol\Theta$ is a $\mathcal C^2$-diffeomorphism onto its image (see, for instance, \cite[Theorem 3.1-1]{Cia00}). Thus, the three vectors
		\[
		\boldsymbol g_i^\varepsilon(x^\varepsilon)
		:=\partial_i^\varepsilon\boldsymbol\Theta(x^\varepsilon)
		\]
		form the covariant basis at $\boldsymbol\Theta(x^\varepsilon)$, while the
		vectors $\boldsymbol g^{j,\varepsilon}$ defined by
		\[
		\boldsymbol g^{j,\varepsilon}\cdot\boldsymbol g_i^\varepsilon
		=\delta_i^j
		\]
		form the corresponding contravariant basis.
		
		The covariant and contravariant components of the metric tensor associated
		with $\boldsymbol\Theta$ and its Christoffel symbols are
		respectively defined by
		\begin{align*}
			g_{ij}^\varepsilon :=
			\boldsymbol g_i^\varepsilon\cdot\boldsymbol g_j^\varepsilon \in \revised{\mathcal C^1}(\overline{\Omega^\varepsilon}),
			\qquad
			g^{ij,\varepsilon}
			:=
			\boldsymbol g^{i,\varepsilon}\cdot\boldsymbol g^{j,\varepsilon}\in \revised{\mathcal C^1}(\overline{\Omega^\varepsilon}),
		\end{align*}
		and
		\begin{align*}
			\Gamma_{ij}^{p,\varepsilon}
			&:=
			\partial_i^\varepsilon\boldsymbol g_j^\varepsilon
			\cdot\boldsymbol g^{p,\varepsilon}
			=\Gamma_{ji}^{p,\varepsilon} \in \revised{\mathcal C^0}(\overline{\Omega^\varepsilon}).
		\end{align*}
		Note that $\Gamma^{3,\varepsilon}_{\alpha3} = \Gamma^{p,\varepsilon}_{33} = 0$. 
		
		The volume element in $\boldsymbol{\Theta}(\overline{\Omega^\varepsilon})$ is given at each point $\boldsymbol{\Theta}(x^\varepsilon)$, $x^\varepsilon \in \overline{\Omega^\varepsilon}$, by $\sqrt{g^\varepsilon(x^\varepsilon)}\,\mathrm{d} x^\varepsilon$, where
		$$
		g^\varepsilon := \operatorname{det}(g^\varepsilon_{ij}) \in \revised{\mathcal{C}^1}(\overline{\Omega^\varepsilon}).
		$$
		
		Given $\boldsymbol v^\varepsilon=(v_i^\varepsilon)\in
		(H^1(\Omega^\varepsilon))^3$, the vector field
		$$
		\widetilde{\boldsymbol v}^{\,\varepsilon}
		:=v_i^\varepsilon\boldsymbol g^{i,\varepsilon}
		$$
		is the associated displacement field of the reference configuration,
		expressed through its covariant components. The covariant components of
		the corresponding linearized strain tensor are
		\begin{equation*}
			\epsilon_{i\|j}^\varepsilon(\boldsymbol v^\varepsilon)
			:=
			\frac12\big(
			\partial_j^\varepsilon v_i^\varepsilon+
			\partial_i^\varepsilon v_j^\varepsilon\big)
			-\Gamma_{ij}^{p,\varepsilon}v_p^\varepsilon
			=
			\epsilon_{j\|i}^\varepsilon(\boldsymbol v^\varepsilon).
		\end{equation*}
		
		\revised{We assume throughout that the shell} is made of a homogeneous, isotropic, linearly magnetoelastic material. Its behavior is thus modeled by Lam\'e constants $\lambda\geq0$ and $\mu>0$, which are independent of both
		$x^\varepsilon$ and $\varepsilon$. The contravariant components of the
		elasticity tensor are then defined as
		\begin{equation*}
			A^{ijk\ell,\varepsilon}
			:=
			\lambda g^{ij,\varepsilon}g^{k\ell,\varepsilon}
			+\mu\big(
			g^{ik,\varepsilon}g^{j\ell,\varepsilon}
			+
			g^{i\ell,\varepsilon}g^{jk,\varepsilon}\big).
		\end{equation*}
		
		Let $\boldsymbol m^\varepsilon=(m_i^\varepsilon)\in
		(H^1(\Omega^\varepsilon))^3$, and define
		\begin{equation*}
			\widetilde{\boldsymbol m}^{\,\varepsilon}
			:=m_i^\varepsilon\boldsymbol g^{i,\varepsilon},
			\qquad
			\widehat{\boldsymbol m}^{\,\varepsilon}
			:=\widetilde{\boldsymbol m}^{\,\varepsilon}
			\circ\boldsymbol\Theta^{-1}
			\quad\text{in }\widehat\Omega^\varepsilon.
		\end{equation*}
		The saturation constraint and the covariant derivatives of the
		magnetization take the form
		\begin{align*}
			&g^{ij,\varepsilon}m_i^\varepsilon m_j^\varepsilon=1
			\quad\text{a.e. in }\Omega^\varepsilon,
		\end{align*}
		and
		\begin{align*}
			&m_{p\|i}^\varepsilon
			:=
			\partial_i^\varepsilon m_p^\varepsilon
			-\Gamma_{ip}^{q,\varepsilon}m_q^\varepsilon.
		\end{align*}
		Consequently,
		\begin{equation*}
			\int_{\widehat\Omega^\varepsilon}
			|\nabla\widehat{\boldsymbol m}^{\,\varepsilon}|^2\,\mathrm{d} \hat{x}^\varepsilon
			=
			\int_{\Omega^\varepsilon}
			g^{ij,\varepsilon}g^{pq,\varepsilon}
			m_{p\|i}^\varepsilon m_{q\|j}^\varepsilon
			\sqrt{g^\varepsilon}\,\mathrm{d} x^\varepsilon.
		\end{equation*}
		
		We denote the covariant components of the magnetic strain contribution by
		\begin{equation*}
			e_{ij}^\varepsilon(\boldsymbol m^\varepsilon)
			:=
			-m_i^\varepsilon m_j^\varepsilon+\frac13g_{ij}^\varepsilon.
		\end{equation*}
		The total elastic strain entering the elastic energy is therefore
		\begin{equation*}
			E_{i\|j}^\varepsilon
			(\boldsymbol v^\varepsilon,\boldsymbol m^\varepsilon)
			:=
			\epsilon_{i\|j}^\varepsilon(\boldsymbol v^\varepsilon)
			+e_{ij}^\varepsilon(\boldsymbol m^\varepsilon).
		\end{equation*}
		This is the curvilinear representation of
		$\boldsymbol\epsilon(\boldsymbol v)+\boldsymbol e(\boldsymbol m)$ in
		\eqref{sec1-eqn1}.
		
		Let
		\[
		L^{1,2}(\mathbb R^3)
		:=
		\big\{\varphi\in L^2_{\mathrm{loc}}(\mathbb R^3):
		\nabla\varphi\in\boldsymbol L^2(\mathbb R^3)\big\}.
		\]
		The magnetostatic potential associated with
		$\boldsymbol m^\varepsilon$ is the function
		$v_{\boldsymbol m^\varepsilon}^\varepsilon\in L^{1,2}(\mathbb R^3)$,
		unique up to an additive constant, such that
		\begingroup\marknewtag
		\begin{equation}
			\label{sec3-eqn1}
			\mu_0\int_{\mathbb R^3}
			\nabla v_{\boldsymbol m^\varepsilon}^\varepsilon
			\cdot\nabla\varphi\,\mathrm{d} \xi
			=
			\int_{\widehat\Omega^\varepsilon}
			\widehat{\boldsymbol m}^{\,\varepsilon}
			\cdot\nabla\varphi\,\mathrm{d} \xi
			\quad\text{for all }\varphi\in L^{1,2}(\mathbb R^3).
		\end{equation}
		\endgroup
		Equivalently,
		\[
		\operatorname{div}\big(
		-\mu_0\nabla v_{\boldsymbol m^\varepsilon}^\varepsilon
		+\chi_{\widehat\Omega^\varepsilon}
		\widehat{\boldsymbol m}^{\,\varepsilon}\big)=0
		\quad\text{in }\mathcal D'(\mathbb R^3).
		\]
		
		We assume that the shell is clamped along
		$\Gamma_0^\varepsilon$, and that it is subjected to applied body forces and magnetic loadings with contravariant components, respectively,
		$f^{i,\varepsilon}\in L^2(\Omega^\varepsilon)$ and $h^{i,\varepsilon}\in L^2(\Omega^\varepsilon)$. Define
		\begin{align*}
			\boldsymbol V(\Omega^\varepsilon)
			&:=
			\big\{\boldsymbol v^\varepsilon=(v_i^\varepsilon)
			\in (H^1(\Omega^\varepsilon))^3:
			\boldsymbol v^\varepsilon=\boldsymbol0
			\text{ on }\Gamma_0^\varepsilon\big\},\\
			\boldsymbol M(\Omega^\varepsilon)
			&:=
			\big\{\boldsymbol m^\varepsilon=(m_i^\varepsilon)
			\in(H^1(\Omega^\varepsilon))^3:
			g^{ij,\varepsilon}m_i^\varepsilon m_j^\varepsilon=1
			\text{ a.e. in }\Omega^\varepsilon\big\},\\
			\boldsymbol{\mathcal A}(\Omega^\varepsilon)
			&:=
			\boldsymbol V(\Omega^\varepsilon)
			\times\boldsymbol M(\Omega^\varepsilon).
		\end{align*}
		For every
		$(\boldsymbol v^\varepsilon,\boldsymbol m^\varepsilon)
		\in\boldsymbol{\mathcal A}(\Omega^\varepsilon)$, let
		\begin{align*}
			\mathcal G^\varepsilon
			(\boldsymbol v^\varepsilon,\boldsymbol m^\varepsilon)
			:={}&
			\frac12\int_{\Omega^\varepsilon}
			A^{ijk\ell,\varepsilon}
			E_{k\|\ell}^\varepsilon
			(\boldsymbol v^\varepsilon,\boldsymbol m^\varepsilon)
			E_{i\|j}^\varepsilon
			(\boldsymbol v^\varepsilon,\boldsymbol m^\varepsilon)
			\sqrt{g^\varepsilon}\,\mathrm{d} x^\varepsilon\\
			&+
			\frac12\int_{\Omega^\varepsilon}
			g^{ij,\varepsilon}g^{pq,\varepsilon}
			m_{p\|i}^\varepsilon m_{q\|j}^\varepsilon
			\sqrt{g^\varepsilon}\,\mathrm{d} x^\varepsilon+
			\frac{\mu_0}{2}\int_{\mathbb R^3}
			|\nabla v_{\boldsymbol m^\varepsilon}^\varepsilon|^2\,\mathrm{d} \xi \\
			&-
			\int_{\Omega^\varepsilon}
			f^{i,\varepsilon}v_i^\varepsilon
			\sqrt{g^\varepsilon}\,\mathrm{d} x^\varepsilon - 
			\int_{\Omega^\varepsilon}
			h^{i,\varepsilon}m_i^\varepsilon
			\sqrt{g^\varepsilon}\,\mathrm{d} x^\varepsilon.
		\end{align*}
		
		We are now in a position to formulate the three-dimensional problem.
		
		\begin{prbs}
			\phantomsection\label{problem3D}
			Find
			$(\boldsymbol u^\varepsilon,\boldsymbol m^\varepsilon)
			\in\boldsymbol{\mathcal A}(\Omega^\varepsilon)$ such that
			\begin{equation*}
				\mathcal G^\varepsilon
				(\boldsymbol u^\varepsilon,\boldsymbol m^\varepsilon)
				=
				\inf_{(\boldsymbol v^\varepsilon,\boldsymbol n^\varepsilon)
					\in\boldsymbol{\mathcal A}(\Omega^\varepsilon)}
				\mathcal G^\varepsilon
				(\boldsymbol v^\varepsilon,\boldsymbol n^\varepsilon).
			\end{equation*}
			\revised{For each admissible magnetization $\boldsymbol n^\varepsilon$,
				the potential $v^\varepsilon_{\boldsymbol n^\varepsilon}$ in
				$\mathcal G^\varepsilon$ is determined, up to an additive constant,
				by the weak Maxwell equation \eqref{sec3-eqn1} with
				$\boldsymbol m^\varepsilon=\boldsymbol n^\varepsilon$. Thus, it is not an independent minimization variable.}
		\end{prbs}
		
		The following existence result is a direct consequence of the $\Gamma$-convergence result in \cite{AKM25}.
		\begin{thm}
			\label{sec3-thm1}
			The minimization problem \hyperref[problem3D]{$\mathcal{P}(\Omega^\varepsilon)$} has at least one solution.
		\end{thm}
		
		The above existence result can be straightforwardly proved by the direct method. However, it should be noted that due to the nonlinearity of the magnetization, the uniqueness of solutions is not expected.

		\section{The scaled 3D problem for a family of linearly elliptic membrane shells}
		\label{sec4}
		
		In Section \ref{sec3}, we considered a general linearly magnetoelastic
		shell. From now on, we restrict ourselves to linearly elliptic membrane
		shells. This means that $\gamma_0=\gamma=\partial\omega$, so the homogeneous boundary condition is imposed on the entire lateral face, and that the middle surface $\boldsymbol\theta(\overline\omega)$ is elliptic in the sense of Section \ref{sec2}. We consider a family of such shells sharing the same middle
		surface, whose thickness $2\varepsilon$ approaches zero.
		
		To effectively perform the asymptotic analysis, we need to specify the \textit{assumptions on the data}. Following the scaling used for elliptic membrane shells in
		\cite[Section 4]{Pie22}, define the fixed domain
		\[
		\Omega:=\omega\times(-1,1),
		\]
		let $x=(x_i)=(y,x_3)$ denote a generic point of $\Omega$, and associate
		with it the point
		\begin{equation*}
			x^\varepsilon 
			:=(y,\varepsilon x_3)\in\Omega^\varepsilon.
		\end{equation*}
		Hence
		\[
		\partial_\alpha^\varepsilon=\partial_\alpha,
		\qquad
		\partial_3^\varepsilon=\frac1\varepsilon\partial_3.
		\]
		
		To the unknown
		$\boldsymbol u^\varepsilon=(u_i^\varepsilon)$, the field
		$\boldsymbol v^\varepsilon=(v_i^\varepsilon)$, and
		$\boldsymbol m^\varepsilon=(m_i^\varepsilon)$ on
		$\Omega^\varepsilon$, we associate the scaled fields on $\Omega$ by
		\begin{equation*}
			u_i(\varepsilon)(x):=u_i^\varepsilon(x^\varepsilon),
			\,
			v_i(x):=v_i^\varepsilon(x^\varepsilon),
			\, \textrm{ and }
			m_i(\varepsilon)(x):=m_i^\varepsilon(x^\varepsilon).
		\end{equation*}
		As in \cite{Pie22}, we assume that there exist functions
		$f^i,h^i\in L^2(\Omega)$, independent of $\varepsilon$, such that
		\begin{equation*}
			f^{i,\varepsilon}(x^\varepsilon)=f^i(x),
			\qquad
			h^{i,\varepsilon}(x^\varepsilon)=h^i(x)
			\quad\text{for a.e. }x\in\Omega.
		\end{equation*}
		
		In view of the proposed scaling, we define
		\[
		\boldsymbol g_i(\varepsilon)(x)
		:=
		\boldsymbol g_i^\varepsilon(x^\varepsilon),
		\qquad
		\boldsymbol g^i(\varepsilon)(x)
		:=
		\boldsymbol g^{i,\varepsilon}(x^\varepsilon),
		\]
		and
		\[
		g_{ij}(\varepsilon)(x)
		:=
		g_{ij}^\varepsilon(x^\varepsilon),
		\qquad
		g^{ij}(\varepsilon)(x)
		:=
		g^{ij,\varepsilon}(x^\varepsilon),
		\qquad
		g(\varepsilon)(x)
		:=
		g^\varepsilon(x^\varepsilon),
		\]
		for each \(x\in\overline{\Omega}\). We also let
		\[
		\Gamma_{ij}^{p}(\varepsilon)(x)
		:=
		\Gamma_{ij}^{p,\varepsilon}(x^\varepsilon),
		\qquad
		A^{ijk\ell}(\varepsilon)(x)
		:=
		A^{ijk\ell,\varepsilon}(x^\varepsilon),
		\quad x\in\overline{\Omega}.
		\]
		
		The scaled linearized elastic strain components in curvilinear coordinates are
		\begin{align*}
			\epsilon_{\alpha\|\beta}(\varepsilon;\boldsymbol v)
			&:=
			\frac12(\partial_\beta v_\alpha+\partial_\alpha v_\beta)
			-\Gamma_{\alpha\beta}^p(\varepsilon)v_p,\\
			\epsilon_{\alpha\|3}(\varepsilon;\boldsymbol v)
			=
			\epsilon_{3\|\alpha}(\varepsilon;\boldsymbol v)
			&:=
			\frac12\left(
			\frac1\varepsilon\partial_3v_\alpha+\partial_\alpha v_3\right)
			-\Gamma_{\alpha3}^{\sigma}(\varepsilon)v_\sigma,\\
			\epsilon_{3\|3}(\varepsilon;\boldsymbol v)
			&:=
			\frac1\varepsilon\partial_3v_3.
		\end{align*}
		Likewise, for a scaled magnetization
		$\boldsymbol m=(m_i)$, define
		\begin{align*}
			m_{p\|\alpha}(\varepsilon;\boldsymbol m)
			&:=
			\partial_\alpha m_p
			-\Gamma_{\alpha p}^{q}(\varepsilon)m_q,\\
			m_{p\|3}(\varepsilon;\boldsymbol m)
			&:=
			\frac1\varepsilon\partial_3m_p
			-\Gamma_{3p}^{q}(\varepsilon)m_q.
		\end{align*}
		The scaled magnetic strain contribution and the total elastic strain are,
		respectively,
		\begin{align*}
			e_{ij}(\varepsilon;\boldsymbol m)
			&:=
			-m_i m_j+\frac13g_{ij}(\varepsilon),\\
			E_{i\|j}
			(\varepsilon;\boldsymbol v,\boldsymbol m)
			&:=
			\epsilon_{i\|j}(\varepsilon;\boldsymbol v)
			+e_{ij}(\varepsilon;\boldsymbol m).
		\end{align*}
		
		Define
		\begin{align*}
			\boldsymbol V(\Omega)
			&:=
			\big\{\boldsymbol v=(v_i)\in(H^1(\Omega))^3:
			\boldsymbol v=\boldsymbol0
			\text{ on }\gamma\times[-1,1]\big\},\\
			\boldsymbol M(\varepsilon;\Omega)
			&:=
			\big\{\boldsymbol m=(m_i)\in(H^1(\Omega))^3:
			g^{ij}(\varepsilon)m_i m_j=1
			\text{ a.e. in }\Omega\big\},\\
			\boldsymbol{\mathcal A}(\varepsilon;\Omega)
			&:=
			\boldsymbol V(\Omega)\times
			\boldsymbol M(\varepsilon;\Omega).
		\end{align*}

		Given $\boldsymbol m\in\boldsymbol M(\varepsilon;\Omega)$, let
		$\widehat{\boldsymbol m}_\varepsilon$ be its representation on the
		reference configuration of the shell, defined by
		\begin{equation}
			\label{sec4-eqn0}
			\widehat{\boldsymbol m}_\varepsilon
			\big(\boldsymbol\Theta(\varepsilon)(x)\big)
			:=
			m_i(x)\boldsymbol g^i(\varepsilon)(x),
		\end{equation}
		\begin{revisedpassage}
			where the scaled parametrization $\boldsymbol\Theta(\varepsilon):\overline\Omega\to\widehat\Omega^\varepsilon$ is defined by
			$$
			\boldsymbol\Theta(\varepsilon)(y,x_3)
			:=\boldsymbol\Theta(y,\varepsilon x_3)
			=\boldsymbol\theta(y)+\varepsilon x_3\boldsymbol a_3(y).
			$$
		\end{revisedpassage}
		Its magnetostatic potential
		$v_{\varepsilon,\boldsymbol m}\in L^{1,2}(\mathbb R^3)$ is characterized
		by
		\begin{align}
			\label{sec4-eqn1}
			\mu_0\int_{\mathbb R^3}
			\nabla v_{\varepsilon,\boldsymbol m}\cdot\nabla\varphi\,\mathrm{d} \xi
			&=
			\int_{\widehat\Omega^\varepsilon}
			\widehat{\boldsymbol m}_\varepsilon\cdot\nabla\varphi\,\mathrm{d} \xi \nonumber\\
			&=
			\varepsilon\int_\Omega
			m_i\boldsymbol g^i(\varepsilon)\cdot
			\big(\nabla\varphi\circ\boldsymbol\Theta(\varepsilon)\big)
			\sqrt{g(\varepsilon)}\,\mathrm{d} x
		\end{align}
		for all $\varphi\in L^{1,2}(\mathbb R^3)$.
		
		We normalize the total energy by the thickness scale and define
		\begin{align*}
			\mathcal G(\varepsilon;
			\boldsymbol v,\boldsymbol m)
			:={}&
			\frac1\varepsilon
			\mathcal G^\varepsilon
			(\boldsymbol v^\varepsilon,\boldsymbol m^\varepsilon)\\
			={}&
			\frac12\int_\Omega
			A^{ijk\ell}(\varepsilon)
			E_{k\|\ell}(\varepsilon;\boldsymbol v,\boldsymbol m)
			E_{i\|j}(\varepsilon;\boldsymbol v,\boldsymbol m)
			\sqrt{g(\varepsilon)}\,\mathrm{d} x\\
			&+
			\frac12\int_\Omega
			g^{ij}(\varepsilon)g^{pq}(\varepsilon)
			m_{p\|i}(\varepsilon;\boldsymbol m)
			m_{q\|j}(\varepsilon;\boldsymbol m)
			\sqrt{g(\varepsilon)}\,\mathrm{d} x\\
			&+
			\frac{\mu_0}{2\varepsilon}\int_{\mathbb R^3}
			|\nabla v_{\varepsilon,\boldsymbol m}|^2\,\mathrm{d} \xi
			-
			\int_\Omega f^iv_i\sqrt{g(\varepsilon)}\,\mathrm{d} x
			-
			\int_\Omega h^im_i\sqrt{g(\varepsilon)}\,\mathrm{d} x.
		\end{align*}

		We are now in a position to introduce the scaled problem of \hyperref[problem3D]{$\mathcal{P}(\Omega^\varepsilon)$}.
		\begin{prbse}
			\phantomsection\label{problem-scaled}
			Find
			$(\boldsymbol u(\varepsilon),\boldsymbol m(\varepsilon))
			\in\boldsymbol{\mathcal A}(\varepsilon;\Omega)$ such that
			\begin{align*}
				\mathcal G\big(
				\varepsilon;\boldsymbol u(\varepsilon),
				\boldsymbol m(\varepsilon)\big)
				=
				\inf_{(\boldsymbol v,\boldsymbol n)
					\in\boldsymbol{\mathcal A}(\varepsilon;\Omega)}
				\mathcal G(\varepsilon;\boldsymbol v,\boldsymbol n).
			\end{align*}
			\revised{For each admissible magnetization $\boldsymbol n$,
				the potential $v_{\varepsilon,\boldsymbol n}$ in
				$\mathcal G(\varepsilon;\boldsymbol v,\boldsymbol n)$ is determined,
				up to an additive constant, by the weak Maxwell equation
				\eqref{sec4-eqn1} with $\boldsymbol m=\boldsymbol n$.}
		\end{prbse}
		
		The next theorem follows immediately from Theorem \ref{sec3-thm1}.
		\begin{thm}
			\label{sec4-thm1}
			The scaled unknown $(\boldsymbol{u}(\varepsilon), \boldsymbol{m}(\varepsilon))$ of a solution $(\boldsymbol{u}^\varepsilon, \boldsymbol{m}^\varepsilon)$ to Problem \hyperref[problem3D]{$\mathcal{P}(\Omega^\varepsilon)$} is a solution to the minimization problem \hyperref[problem-scaled]{$\mathcal{P}(\varepsilon;\Omega)$}.
		\end{thm}

		For later use, we collect the standard asymptotic properties of the
		geometrical quantities. The notation $O(\varepsilon^r)$ denotes a
		remainder of order $\varepsilon^r$ in the supremum norm over
		$\overline\Omega$.
		
		\begin{lem} 
			\label{sec4:lem1}
			Let $\varepsilon_0$ be defined as in \cite[Theorem 3.1-1]{Cia00}. The vector fields $\boldsymbol{g}_i (\varepsilon)$ and $\boldsymbol{g}^j(\varepsilon)$ have the following properties:
			\begin{align*}
				\boldsymbol{g}_\alpha (\varepsilon) &= \boldsymbol{a}_\alpha - \varepsilon x_3 b^\sigma_\alpha \boldsymbol{a}_\sigma, \quad \boldsymbol{g}_3 (\varepsilon) = \boldsymbol{a}_3 , \\
				\boldsymbol{g}^\alpha (\varepsilon) &= \boldsymbol{a}^\alpha + \varepsilon x_3 b^\alpha_\sigma \boldsymbol{a}^\sigma + O(\varepsilon^2), \quad \boldsymbol{g}^3(\varepsilon) = \boldsymbol{a}^3.
			\end{align*}
			
			The functions $A^{ijk\ell} (\varepsilon) = A^{jik\ell} (\varepsilon) = A^{k\ell ij} (\varepsilon)$ have the following properties:
			\[
			A^{ijk\ell} (\varepsilon) = A^{ijk\ell} ( 0 ) + O(\varepsilon) , \quad A^{\alpha \beta \sigma 3} (\varepsilon) = A^{\alpha 333} (\varepsilon ) = 0,
			\]
			where
			\begin{align*}
				A^{\alpha \beta \sigma \tau} (0) &= \lambda a^{\alpha \beta} a^{\sigma \tau} + \mu (a^{\alpha \sigma}a^{\beta \tau} + a^{\alpha \tau} a^{\beta \sigma}), \\
				A^{\alpha \beta 33} (0) &= \lambda a^{\alpha \beta} , \quad A^{\alpha 3 \sigma 3} (0) = \mu a^{\alpha \sigma}, \quad A^{3333}(0) = \lambda + 2 \mu. 
			\end{align*}
			
			The functions $\Gamma^p_{ij} (\varepsilon)$ and $g(\varepsilon)$ have the following properties:
			\begin{align*}
				\Gamma^\sigma_{\alpha \beta} (\varepsilon) &= \Gamma^\sigma_{\alpha \beta} - \varepsilon x_3 (\partial_\alpha b^\sigma_\beta + \Gamma^\sigma_{\alpha \tau} b^\tau_\beta - \Gamma^\tau_{\alpha \beta} b^\sigma_\tau) + O(\varepsilon^2), \\
				\Gamma^3_{\alpha \beta} (\varepsilon) &= b_{\alpha \beta} - \varepsilon x_3 b^\sigma_\alpha b_{\sigma \beta}, \quad \partial_3 \Gamma^p_{\alpha \beta} (\varepsilon) = O(\varepsilon), \\
				\Gamma^\sigma_{\alpha 3} (\varepsilon) &= - b^\sigma_\alpha - \varepsilon x_3 b^\tau_\alpha b^\sigma_\tau + O(\varepsilon^2), \quad \Gamma^3_{\alpha 3} (\varepsilon) = \Gamma^p_{33} (\varepsilon ) =0, \\
				g(\varepsilon) &= a + O(\varepsilon).
			\end{align*}
			
			In particular, then, there exist constants $g_0>0$, $g_1$, and $C_0$, such that
			\[
			g_0 \leq g(\varepsilon)(x)  \leq g_1
			\textup{ \ and \ } 
			\sum_{i, j} \left| t_{ij} \right|^2 \leq C_0 A^{ijk\ell} (\varepsilon) (x) t_{k\ell} t_{ij}
			\]
			for all $0<\varepsilon\leq \varepsilon_0$, all $x \in \overline{\Omega}$, and all symmetric matrices $(t_{ij})$.
		\end{lem}
		
		\begin{proof}
			We refer the reader to the proofs of Theorems 3.3-1 and 3.3-2 in \cite{Cia00}. The details are omitted.
		\end{proof}
		
		The above properties are purely geometrical and hold for general
		shell families, i.e., the ellipticity of the middle surface is not required. However, it becomes essential in the following Korn inequality (see, e.g., \cite[Theorem 4.3]{Pie22}).
		
		\begin{thm}
			\label{sec4-thm2}
			There exist constants $\varepsilon_1>0$ and $C_1>0$ such that
			\[
			\left\{
			\sum_\alpha\|v_\alpha\|_{H^1(\Omega)}^2
			+\|v_3\|_{L^2(\Omega)}^2
			\right\}^{1/2}
			\leq
			C_1
			\left\{
			\sum_{i,j}
			\|\epsilon_{i\|j}(\varepsilon;\boldsymbol v)\|_{L^2(\Omega)}^2
			\right\}^{1/2}
			\]
			for every $0<\varepsilon\leq\varepsilon_1$ and every
			$\boldsymbol v\in\boldsymbol V(\Omega)$. Here, the constant $C_1$ is independent of the thickness $2\varepsilon$.
		\end{thm}

		\section{Proof of \texorpdfstring{$\Gamma$}{Gamma}-convergence}
		\label{sec5}
		
		As mentioned in Section \ref{sec1}, we will use the notion of $\Gamma$-convergence to study the asymptotic behavior of the three-dimensional problem \hyperref[problem-scaled]{$\mathcal{P}(\varepsilon;\Omega)$} as the thickness $2\varepsilon$ goes to zero. The general idea behind $\Gamma$-convergence of a sequence of functionals $\mathcal{F}_\varepsilon$ to $\mathcal{F}_0$ (see, for example, \cite{Bra02}) can be understood as follows:
		\begin{enumerate}
			\item \textbf{Equi-coercivity:} For any sequence \( (y_\varepsilon)_\varepsilon \) such that \( \sup_\varepsilon \mathcal{F}_\varepsilon(y_\varepsilon) < \infty \), there exists some converging subsequence (possibly after rescaling), with a limit \( y \).
			\item \textbf{Liminf inequality:} For any sequence \( (y_\varepsilon)_\varepsilon \) converging in the above (possibly rescaled) sense to \( y \), we have \( \liminf_{\varepsilon \to 0} \mathcal{F}_\varepsilon(y_\varepsilon) \geq \mathcal{F}(y) \).
			\item \textbf{Recovery sequence:} For any \( y \), there exists an approximate sequence \( (y_\varepsilon)_\varepsilon \) converging to \( y \) such that \( \limsup_{\varepsilon \to 0} \mathcal{F}_\varepsilon(y_\varepsilon) \leq \mathcal{F}(y) \).
		\end{enumerate}
		
		To define the limit two-dimensional problem, we need to introduce the following quantities on the middle surface $\boldsymbol{\theta}(\overline{\omega})$. First, for $\boldsymbol m=(m_i)\in (H^1(\omega))^3$, define the associated magnetization vector field on the middle surface by
		$$
		\boldsymbol m_S :=m_i\boldsymbol a^i.
		$$
		Its saturation constraint is
		\[
		|\boldsymbol m_S|^2
		=a^{\alpha\beta}m_\alpha m_\beta+m_3^2=1
		\text{ a.e. in }\omega.
		\]
		Its surface gradient is given by
		\begin{equation}
			\label{sec5-eqn1}
			\nabla_S\boldsymbol m_S := \partial_\alpha\boldsymbol m_S\otimes\boldsymbol a^\alpha.
		\end{equation}
		More explicitly,
		\[
		\partial_\alpha\boldsymbol m_S =
		m_{\sigma\mid\alpha}\boldsymbol a^\sigma
		+m_{3\mid\alpha}\boldsymbol a^3,
		\]
		where
		\[
		m_{\sigma\mid\alpha} := \partial_\alpha m_\sigma -\Gamma_{\alpha\sigma}^{\tau}m_\tau -b_{\alpha\sigma}m_3,
		\qquad
		m_{3\mid\alpha} :=\partial_\alpha m_3+b_\alpha^\sigma m_\sigma.
		\]
		Consequently,
		\begin{equation}
			\label{sec5-eqn2}
			|\nabla_S\boldsymbol m_S|^2=a^{\alpha\beta}a^{\sigma\tau}m_{\sigma\mid\alpha}m_{\tau\mid\beta}+a^{\alpha\beta} m_{3\mid\alpha}m_{3\mid\beta}.
		\end{equation}
		
		We define
		\begin{align*}
			\boldsymbol M_S(\omega)&:=
			\left\{\boldsymbol m=(m_i)\in\ (H^1(\omega))^3:
			a^{\alpha\beta}m_\alpha m_\beta+m_3^2=1
			\text{ a.e. in }\omega\right\},\\
			\boldsymbol{\mathcal A}_M(\omega)&:=
			\boldsymbol V_M(\omega)\times\boldsymbol M_S(\omega).
		\end{align*}
		where $\boldsymbol{V}_M(\omega)$ is given in Theorem \ref{sec2-thm1}. As in the classical linear membrane shell theory, we define the two-dimensional
		elasticity tensor
		\begin{equation*}
			a^{\alpha\beta\sigma\tau}
			:=
			\frac{4\lambda\mu}{\lambda+2\mu}
			a^{\alpha\beta}a^{\sigma\tau}
			+2\mu\left(
			a^{\alpha\sigma}a^{\beta\tau}
			+a^{\alpha\tau}a^{\beta\sigma}\right).
		\end{equation*}
		
		For $(\boldsymbol\eta,\boldsymbol m)
		\in\boldsymbol{\mathcal A}_M(\omega)$, let
		\begin{align*}
			e_{\alpha\beta}(\boldsymbol m)
			&:=
			-m_\alpha m_\beta+\frac13a_{\alpha\beta},\\
			\mathcal E_{\alpha\beta}
			(\boldsymbol\eta,\boldsymbol m)
			&:=
			\gamma_{\alpha\beta}(\boldsymbol\eta)
			+e_{\alpha\beta}(\boldsymbol m).
		\end{align*}
		
		We are now in a position to define the two-dimensional problem \hyperref[membrane]{$\mathcal{P}_M(\omega)$}.
		\begin{prbsm}
			\phantomsection\label{membrane}
			Find $(\boldsymbol\zeta,\boldsymbol m)
			\in\boldsymbol{\mathcal A}_M(\omega)$ such that
			\begin{equation*}
				\mathcal G_M(\boldsymbol\zeta,\boldsymbol m)
				=
				\inf_{(\boldsymbol\eta,\boldsymbol n)
					\in\boldsymbol{\mathcal A}_M(\omega)}
				\mathcal G_M(\boldsymbol\eta,\boldsymbol n),
			\end{equation*}
			where
			\begin{align*}
				\mathcal G_M(\boldsymbol\eta,\boldsymbol m)
				:={}&
				\frac12\int_\omega
				a^{\alpha\beta\sigma\tau}
				\mathcal E_{\sigma\tau}(\boldsymbol\eta,\boldsymbol m)
				\mathcal E_{\alpha\beta}(\boldsymbol\eta,\boldsymbol m)
				\sqrt a\,\mathrm{d} y\\
				&+
				\int_\omega|\nabla_S\boldsymbol m_S|^2\sqrt a\,\mathrm{d} y
				+\frac1{\mu_0}\int_\omega m_3^2\sqrt a\,\mathrm{d} y
				-\int_\omega p^i\eta_i\sqrt a\,\mathrm{d} y
				-\int_\omega q^im_i\sqrt a\,\mathrm{d} y
			\end{align*}
			for each $(\boldsymbol\eta,\boldsymbol m) \in \boldsymbol{\mathcal A}_M(\omega)$, and
			\begin{align*}
				p^i(y):=\int_{-1}^{1}f^i(y,x_3)\,\mathrm{d} x_3, \quad
				q^i(y):=\int_{-1}^{1}h^i(y,x_3)\,\mathrm{d} x_3.
			\end{align*}
		\end{prbsm}
		
		We now formulate the main result on a fixed space. Let
		\[
		\boldsymbol{\mathcal X}
		:=
		(L^2(\Omega))^3\times(L^2(\Omega))^3,
		\]
		endowed with its weak topology. For each $\varepsilon>0$, extend the
		scaled functional to $\boldsymbol{\mathcal X}$ by setting
		\begin{equation*}
			\overline{\mathcal G}_\varepsilon
			(\boldsymbol v,\boldsymbol m)
			:=
			\begin{cases}
				\mathcal G(\varepsilon;\boldsymbol v,\boldsymbol m),
				&\text{if }(\boldsymbol v,\boldsymbol m)
				\in\boldsymbol{\mathcal A}(\varepsilon;\Omega),\\
				+\infty,&\text{otherwise}.
			\end{cases}
		\end{equation*}
		We identify each function on $\omega$ with its constant extension in the
		transverse variable $x_3$. The limiting functional
		$\overline{\mathcal G}_0:\boldsymbol{\mathcal X}
		\to(-\infty,+\infty]$ is defined by
		\begin{equation*}
			\overline{\mathcal G}_0(\boldsymbol v,\boldsymbol m)
			:=
			\begin{cases}
				\mathcal G_M(\boldsymbol\eta,\boldsymbol n),
				&\begin{aligned}
					&\text{if }\boldsymbol v(y,x_3)=\boldsymbol\eta(y),\\[-1mm]
					&\phantom{\text{if }}\boldsymbol m(y,x_3)=\boldsymbol n(y)
					\text{ and }
					(\boldsymbol\eta,\boldsymbol n)
					\in\boldsymbol{\mathcal A}_M(\omega),
				\end{aligned}\\
				+\infty,&\text{otherwise}.
			\end{cases}
		\end{equation*}
		
		\begin{thm}[$\Gamma$-Convergence]
			\label{sec5-thm1}
			The family
			$(\overline{\mathcal G}_\varepsilon)_{\varepsilon>0}$ is
			equi-coercive and sequentially $\Gamma$-converges to
			$\overline{\mathcal G}_0$ with respect to the weak topology of
			$\boldsymbol{\mathcal X}$. More precisely, the following assertions
			hold.
			\begin{enumerate}
				\item If
				\[
				\sup_{\varepsilon>0}
				\overline{\mathcal G}_\varepsilon
				(\boldsymbol v(\varepsilon),\boldsymbol m(\varepsilon))
				<+\infty,
				\]
				then there exists a subsequence, not relabeled, and
				$(\boldsymbol\eta,\boldsymbol m)
				\in\boldsymbol{\mathcal A}_M(\omega)$ such that
				\begin{align}
					\label{sec5-eqn3}
					v_\alpha(\varepsilon)&\rightharpoonup\eta_\alpha
					&&\text{weakly in }H^1(\Omega),\nonumber\\
					v_3(\varepsilon)&\rightharpoonup\eta_3
					&&\text{weakly in }L^2(\Omega),\\
					m_i(\varepsilon)&\rightharpoonup m_i
					&&\text{weakly in }H^1(\Omega),\nonumber\\
					m_i(\varepsilon)&\to m_i
					&&\text{strongly in }L^2(\Omega).\nonumber
				\end{align}
				All the limits in \eqref{sec5-eqn3} are independent of
				$x_3$.
				
				\item For every sequence
				$(\boldsymbol v(\varepsilon),\boldsymbol m(\varepsilon))$
				converging weakly in $\boldsymbol{\mathcal X}$ to
				$(\boldsymbol v,\boldsymbol m)$,
				\begin{equation}
					\label{sec5-eqn4}
					\overline{\mathcal G}_0(\boldsymbol v,\boldsymbol m)
					\leq
					\liminf_{\varepsilon\to0}
					\overline{\mathcal G}_\varepsilon
					(\boldsymbol v(\varepsilon),\boldsymbol m(\varepsilon)).
				\end{equation}
				
				\item For every
				$(\boldsymbol v,\boldsymbol m)\in\boldsymbol{\mathcal X}$, there exists
				a sequence
				$(\boldsymbol v(\varepsilon),\boldsymbol m(\varepsilon))$
				converging strongly in $\boldsymbol{\mathcal X}$ to
				$(\boldsymbol v,\boldsymbol m)$ such that
				\begin{equation}
					\label{sec5-eqn5}
					\overline{\mathcal G}_0(\boldsymbol v,\boldsymbol m)
					\geq
					\limsup_{\varepsilon\to0}
					\overline{\mathcal G}_\varepsilon
					(\boldsymbol v(\varepsilon),\boldsymbol m(\varepsilon)).
				\end{equation}
			\end{enumerate}
		\end{thm}
		
		The proof of Theorem \ref{sec5-thm1} is divided into the following
		lemmas. The first one provides the necessary compactness.
		
		\begin{lem}[Compactness]
			\label{sec5-lem1}
			Let
			$(\boldsymbol v(\varepsilon),\boldsymbol m(\varepsilon))
			\in\boldsymbol{\mathcal A}(\varepsilon;\Omega)$ satisfy
			\[
			\sup_{\varepsilon>0}
			\mathcal G(\varepsilon;
			\boldsymbol v(\varepsilon),\boldsymbol m(\varepsilon))
			<+\infty.
			\]
			Then there exists a subsequence and
			$(\boldsymbol\eta,\boldsymbol m)
			\in\boldsymbol{\mathcal A}_M(\omega)$ for which all the convergences in
			\eqref{sec5-eqn3} hold.
		\end{lem}
		
		\begin{proof}
			The saturation constraint, the uniform bounds for
			$g^{ij}(\varepsilon)$, and Lemma \ref{sec4:lem1} imply
			\[
			\|m_i(\varepsilon)\|_{L^\infty(\Omega)}\leq C.
			\]
			Hence the components
			$e_{ij}(\varepsilon;\boldsymbol m(\varepsilon))$ are uniformly bounded
			in $L^\infty(\Omega)$. By the uniform ellipticity of
			$A^{ijk\ell}(\varepsilon)$, Young's inequality, and Theorem
			\ref{sec4-thm2}, the elastic energy and the mechanical loading term give
			\begin{align*}
				&\sum_\alpha
				\|v_\alpha(\varepsilon)\|_{H^1(\Omega)}^2
				+\|v_3(\varepsilon)\|_{L^2(\Omega)}^2\\
				&\qquad\leq
				C\left(
				1+\int_\Omega
				A^{ijk\ell}(\varepsilon)
				E_{k\|\ell}(\varepsilon)
				E_{i\|j}(\varepsilon)
				\sqrt{g(\varepsilon)}\,\mathrm{d} x\right).
			\end{align*}
			Thanks to the saturation constraint, the magnetic loading term is uniformly bounded. The exchange energy, together with the uniform
			bounds for the Christoffel symbols, then yields
			\begin{align}
				\label{sec5-eqn6}
				&\sum_\alpha
				\|v_\alpha(\varepsilon)\|_{H^1(\Omega)}^2
				+\|v_3(\varepsilon)\|_{L^2(\Omega)}^2\nonumber\\
				&\quad+
				\sum_i\left(
				\|m_i(\varepsilon)\|_{H^1(\Omega)}^2+
				\frac1{\varepsilon^2}
				\|\partial_3m_i(\varepsilon)\|_{L^2(\Omega)}^2
				\right)\leq C.
			\end{align}
			Consequently, after extracting a subsequence, the weak convergences in
			\eqref{sec5-eqn3} hold. The compact embedding
			$H^1(\Omega)\hookrightarrow L^2(\Omega)$ gives the strong convergence
			of the magnetization.
			
			Estimate \eqref{sec5-eqn6} implies $\partial_3m_i=0$. The boundedness
			of the scaled strains also gives $\partial_3\eta_i=0$ in the sense of
			distributions. Indeed,
			\[
			\partial_3v_3(\varepsilon)
			=
			\varepsilon\epsilon_{3\|3}
			(\varepsilon;\boldsymbol v(\varepsilon)),
			\]
			whereas
			\[
			\partial_3v_\alpha(\varepsilon)
			=
			2\varepsilon\epsilon_{\alpha\|3}
			(\varepsilon;\boldsymbol v(\varepsilon))
			-\varepsilon\partial_\alpha v_3(\varepsilon)
			+2\varepsilon\Gamma_{\alpha3}^{\sigma}(\varepsilon)
			v_\sigma(\varepsilon).
			\]
			The first right-hand side tends to zero in $L^2(\Omega)$ and the second
			tends to zero in $\mathcal D'(\Omega)$.
			
			Finally, the strong $L^2$ convergence of the magnetization and
			$g^{ij}(\varepsilon)\to g^{ij}(0)$ uniformly give
			\[
			a^{\alpha\beta}m_\alpha m_\beta+m_3^2=1
			\quad\text{a.e. in }\omega.
			\]
			The trace condition passes to the weak limit for
			$\eta_\alpha$. Therefore
			$(\boldsymbol\eta,\boldsymbol m)
			\in\boldsymbol{\mathcal A}_M(\omega)$.
		\end{proof}

		\begin{lem}[Limit of the magnetic energies]
			\label{sec5-lem2}
			Let
			$(\boldsymbol v(\varepsilon),\boldsymbol m(\varepsilon))
			\in\boldsymbol{\mathcal A}(\varepsilon;\Omega)$ as in Lemma \ref{sec5-lem1}, and set
			\[
			\boldsymbol{\mathfrak m}_\varepsilon(x)
			:=
			m_i(\varepsilon)(x)\boldsymbol g^i(\varepsilon)(x).
			\]
			Then, up to a subsequence (not relabeled),
			$\boldsymbol{\mathfrak m}_\varepsilon$ converges weakly in $(H^1(\Omega))^3$ to $\boldsymbol m_S = m_i \boldsymbol{a}^i$, 
			\begin{align}
				\label{sec5-eqn7}
				&\liminf_{\varepsilon\to0}
				\frac12\int_\Omega
				g^{ij}(\varepsilon)g^{pq}(\varepsilon)
				m_{p\|i}(\varepsilon;\boldsymbol m(\varepsilon))
				m_{q\|j}(\varepsilon;\boldsymbol m(\varepsilon))
				\sqrt{g(\varepsilon)}\,\mathrm{d} x\nonumber\\
				&\hspace{35mm}\geq
				\int_\omega|\nabla_S\boldsymbol m_S|^2\sqrt a\,\mathrm{d} y,
			\end{align}
			and
			\begin{equation}
				\label{sec5-eqn8}
				\lim_{\varepsilon\to0}
				\frac{\mu_0}{2\varepsilon}
				\int_{\mathbb R^3}
				|\nabla v_{\varepsilon,\boldsymbol m(\varepsilon)}|^2\,\mathrm{d} \xi
				=
				\frac1{\mu_0}\int_\omega m_3^2\sqrt a\,\mathrm{d} y.
			\end{equation}
			Conversely, for each
			$\boldsymbol m\in\boldsymbol M_S(\omega)$, define $\boldsymbol{m}^\varepsilon = (m_i^\varepsilon)$ by
			\begin{equation}
				\label{sec5-eqn9}
				m_i^\varepsilon(y,x_3)
				:=
				\boldsymbol m_S(y)\cdot
				\boldsymbol g_i(\varepsilon)(y,x_3).
			\end{equation}
			Then
			$\boldsymbol m^\varepsilon
			\in\boldsymbol M(\varepsilon;\Omega)$,
			$\boldsymbol m^\varepsilon\to\boldsymbol m$ strongly in
			$(H^1(\Omega))^3$, and equality is attained in
			\eqref{sec5-eqn7} in the limit; moreover,
			\eqref{sec5-eqn8} holds for this sequence.
		\end{lem}
		
		\begin{proof}
			For the reader's convenience, we divide the proof into three parts.
			
			\smallskip
			
			\noindent (i) The inequality \eqref{sec5-eqn7} is a consequence of Lemmas \ref{sec5-lem1} and \ref{sec4:lem1} together with identities \eqref{sec5-eqn1}--\eqref{sec5-eqn2}.
			
			\smallskip
			
			\noindent (ii) We next prove \eqref{sec5-eqn8}. Set
			\[
			\boldsymbol d_\varepsilon
			:=-\mu_0\nabla
			v_{\varepsilon,\boldsymbol m(\varepsilon)}.
			\]
			We adapt the method of the proof of \cite[Lemma~1]{DiF20}.
			Testing the Maxwell equation \eqref{sec4-eqn1} with
			$v_{\varepsilon,\boldsymbol m(\varepsilon)}$, we deduce that
			\[
			\|\boldsymbol d_\varepsilon\|_{L^2(\mathbb R^3)}^2
			=-\int_{\widehat\Omega^\varepsilon}
			\widehat{\boldsymbol m}_\varepsilon\cdot
			\boldsymbol d_\varepsilon\,\mathrm{d} \xi,
			\]
			where $\widehat{\boldsymbol m}_\varepsilon$ is given by \eqref{sec4-eqn0}. H\"older's inequality then implies
			\[
			\|\boldsymbol d_\varepsilon\|_{L^2(\mathbb R^3)}
			\leq
			\|\widehat{\boldsymbol m}_\varepsilon
			\|_{L^2(\widehat\Omega^\varepsilon)}.
			\]
			Since $|\widehat{\boldsymbol m}_\varepsilon|=1$ almost everywhere,
			it follows that
			\[
			\frac1\varepsilon
			\|\boldsymbol d_\varepsilon\|_{L^2(\mathbb R^3)}^2
			\leq
			\frac1\varepsilon|\widehat\Omega^\varepsilon|
			=\int_\Omega\sqrt{g(\varepsilon)}\,\mathrm{d} x
			\leq C.
			\]
			
			Let
			$\varepsilon_0$ be given in Lemma \ref{sec4:lem1}. Define
			\[
			J_\varepsilon
			:=\left(-\frac{\varepsilon_0}{\varepsilon},
			\frac{\varepsilon_0}{\varepsilon}\right),
			\qquad
			\boldsymbol\Psi_\varepsilon(y,s)
			:=\boldsymbol\theta(y)+\varepsilon s\boldsymbol a_3(y),
			\quad (y,s)\in\omega\times J_\varepsilon.
			\]
			We use the same notation
			$\boldsymbol g_i(\varepsilon;y,s)$,
			$\boldsymbol g^i(\varepsilon;y,s)$, and
			$g(\varepsilon;y,s)$ for the natural extensions of the geometrical
			coefficients from $(-1,1)$ to $J_\varepsilon$. The change of variables
			$\xi=\boldsymbol\Psi_\varepsilon(y,s)$ gives
			\[
			\mathrm{d} \xi
			=\varepsilon\sqrt{g(\varepsilon;y,s)}\,
			\mathrm{d} y\,\mathrm{d} s.
			\]
			On every fixed set $\omega\times(-R,R)$,
			\[
			\boldsymbol g_i(\varepsilon;y,s)\to\boldsymbol a_i(y),
			\qquad
			\boldsymbol g^i(\varepsilon;y,s)\to\boldsymbol a^i(y),
			\qquad
			\sqrt{g(\varepsilon;y,s)}\to\sqrt{a(y)}
			\]
			uniformly as $\varepsilon\to0$.
			
			Let
			\[
			\boldsymbol D_\varepsilon(y,s)
			:=\boldsymbol d_\varepsilon
			\big(\boldsymbol\Psi_\varepsilon(y,s)\big),
			\qquad (y,s)\in\omega\times J_\varepsilon,
			\]
			and extend $\boldsymbol D_\varepsilon$ by zero to
			$\omega\times\mathbb R$. The preceding estimate, the change of
			variables, and the uniform lower bound for the volume density give
			\[
			\int_{\omega\times\mathbb R}
			|\boldsymbol D_\varepsilon|^2\,
			\mathrm{d} y\,\mathrm{d} s
			\leq
			\frac C\varepsilon
			\|\boldsymbol d_\varepsilon\|_{L^2(\mathbb R^3)}^2
			\leq C.
			\]
			Thus, up to a subsequence,
			\[
			\boldsymbol D_\varepsilon
			\rightharpoonup\boldsymbol D_0
			\quad\text{weakly in }
			(L^2(\omega\times\mathbb R))^3.
			\]
			
			We first identify the tangential component of the limit. Let
			\[
			w_\varepsilon
			:=v_{\varepsilon,\boldsymbol m(\varepsilon)}
			\circ\boldsymbol\Psi_\varepsilon.
			\]
			Since $\boldsymbol d_\varepsilon=-\mu_0\nabla
			v_{\varepsilon,\boldsymbol m(\varepsilon)}$, we have
			\[
			\boldsymbol D_\varepsilon\cdot
			\boldsymbol g_\alpha(\varepsilon)
			=-\mu_0\partial_\alpha w_\varepsilon,
			\qquad
			\boldsymbol D_\varepsilon\cdot\boldsymbol a_3
			=-\frac{\mu_0}{\varepsilon}\partial_s w_\varepsilon.
			\]
			Equality of the mixed derivatives of $w_\varepsilon$ therefore gives
			\[
			\partial_s\big(
			\boldsymbol D_\varepsilon\cdot
			\boldsymbol g_\alpha(\varepsilon)\big)
			=\varepsilon\partial_\alpha\big(
			\boldsymbol D_\varepsilon\cdot\boldsymbol a_3\big)
			\quad\text{in }\mathcal D'(\omega\times J_\varepsilon).
			\]
			For any test function compactly supported in
			$\omega\times\mathbb R$, its support is contained in
			$\omega\times J_\varepsilon$ for all sufficiently small
			$\varepsilon$. Passing to the limit in the above identity gives
			\[
			\partial_s(\boldsymbol D_0\cdot\boldsymbol a_\alpha)=0
			\quad\text{in }\mathcal D'(\omega\times\mathbb R).
			\]
			Hence $\boldsymbol D_0\cdot\boldsymbol a_\alpha$ is independent of
			$s$. Since this function belongs to
			$L^2(\omega\times\mathbb R)$, it must be zero. Therefore
			\[
			\boldsymbol D_0\cdot\boldsymbol a_\alpha=0,
			\qquad \alpha=1,2.
			\]
			
			We now identify the normal component. The Maxwell equation can be
			written as
			\[
			\operatorname{div}\big(
			\boldsymbol d_\varepsilon
			+\chi_{\widehat\Omega^\varepsilon}
			\widehat{\boldsymbol m}_\varepsilon\big)=0
			\quad\text{in }\mathcal D'(\mathbb R^3).
			\]
			Let $\psi\in\mathcal D(\omega\times\mathbb R)$. For sufficiently
			small $\varepsilon$, define the test function
			$\varphi_\varepsilon$ by
			\[
			\varphi_\varepsilon
			\big(\boldsymbol\Psi_\varepsilon(y,s)\big)
			:=\varepsilon\psi(y,s).
			\]
			Since $\psi$ has compact support, this function can be extended by zero
			to a function in $L^{1,2}(\mathbb R^3)$. Moreover,
			\[
			\nabla\varphi_\varepsilon
			\circ\boldsymbol\Psi_\varepsilon
			=\varepsilon\partial_\alpha\psi\,
			\boldsymbol g^\alpha(\varepsilon)
			+\partial_s\psi\,\boldsymbol a_3.
			\]
			Substituting this test function in the weak Maxwell equation, changing
			variables, and dividing by $\varepsilon$ yield
			\begin{align*}
				0={}&\int_{\omega\times\mathbb R}
				\Big[
				\boldsymbol D_\varepsilon
				+\chi_{(-1,1)}(s)
				\boldsymbol{\mathfrak m}_\varepsilon(y,s)
				\Big]\cdot\boldsymbol a_3\,
				\partial_s\psi\sqrt{g(\varepsilon;y,s)}\,
				\mathrm{d} y\,\mathrm{d} s\\
				&+\varepsilon\int_{\omega\times\mathbb R}
				\Big[
				\boldsymbol D_\varepsilon
				+\chi_{(-1,1)}(s)
				\boldsymbol{\mathfrak m}_\varepsilon(y,s)
				\Big]\cdot\boldsymbol g^\alpha(\varepsilon;y,s)\,
				\partial_\alpha\psi\sqrt{g(\varepsilon;y,s)}\,
				\mathrm{d} y\,\mathrm{d} s.
			\end{align*}
			Here $\boldsymbol{\mathfrak m}_\varepsilon$ is extended by zero for
			$s\notin(-1,1)$. The second integral tends to zero. In the first one,
			we use the weak convergence of $\boldsymbol D_\varepsilon$, the strong
			$L^2$ convergence of $\boldsymbol{\mathfrak m}_\varepsilon$ to
			$\boldsymbol m_S$, and the uniform convergence of the geometrical
			coefficients on the support of $\psi$. We obtain
			\[
			\int_{\omega\times\mathbb R}
			\big(
			\boldsymbol D_0\cdot\boldsymbol a_3
			+m_3\chi_{(-1,1)}
			\big)\partial_s\psi\sqrt a\,
			\mathrm{d} y\,\mathrm{d} s=0.
			\]
			Thus
			$\boldsymbol D_0\cdot\boldsymbol a_3
			+m_3\chi_{(-1,1)}$ is independent of $s$. It belongs to
			$L^2(\omega\times\mathbb R)$ and is consequently zero. Combining the
			tangential and normal identities gives
			\[
			\boldsymbol D_0(y,s)
			=-m_3(y)\boldsymbol a_3(y)\chi_{(-1,1)}(s).
			\]
			Since every weakly convergent subsequence converges to the same limit,
			the whole sequence converges weakly to $\boldsymbol D_0$.
			
			Finally, testing \eqref{sec4-eqn1} by
			$v_{\varepsilon,\boldsymbol m(\varepsilon)}$ and using
			$\boldsymbol d_\varepsilon=-\mu_0\nabla
			v_{\varepsilon,\boldsymbol m(\varepsilon)}$, we obtain
			\[
			\frac{\mu_0}{2\varepsilon}
			\int_{\mathbb R^3}
			|\nabla v_{\varepsilon,\boldsymbol m(\varepsilon)}|^2
			\,\mathrm{d} \xi
			=-\frac1{2\mu_0\varepsilon}
			\int_{\widehat\Omega^\varepsilon}
			\widehat{\boldsymbol m}_\varepsilon\cdot
			\boldsymbol d_\varepsilon\,\mathrm{d} \xi.
			\]
			Changing variables in the integral on the right-hand side gives
			\[
			\frac{\mu_0}{2\varepsilon}
			\int_{\mathbb R^3}
			|\nabla v_{\varepsilon,\boldsymbol m(\varepsilon)}|^2
			\,\mathrm{d} \xi
			=-\frac1{2\mu_0}
			\int_\Omega
			\boldsymbol{\mathfrak m}_\varepsilon\cdot
			\boldsymbol D_\varepsilon
			\sqrt{g(\varepsilon)}\,\mathrm{d} x.
			\]
			We can now pass to the limit to obtain
			\begin{align*}
				\lim_{\varepsilon\to0}
				\frac{\mu_0}{2\varepsilon}
				\int_{\mathbb R^3}
				|\nabla v_{\varepsilon,\boldsymbol m(\varepsilon)}|^2
				\,\mathrm{d} \xi
				&=-\frac1{2\mu_0}
				\int_{\omega\times(-1,1)}
				\boldsymbol m_S\cdot
				\big(-m_3\boldsymbol a_3\big)\sqrt a\,
				\mathrm{d} y\,\mathrm{d} s\\
				&=\frac1{\mu_0}
				\int_\omega m_3^2\sqrt a\,\mathrm{d} y,
			\end{align*}
			and \eqref{sec5-eqn8} follows. 
			
			\smallskip
			
			\noindent (iii) It remains to verify the recovery sequence. Fix
			$\boldsymbol m\in\boldsymbol M_S(\omega)$ and let
			$\boldsymbol m^\varepsilon$ be defined by \eqref{sec5-eqn9}. More
			explicitly, for $(y,x_3)\in\Omega$,
			\[
			m_i^\varepsilon(y,x_3)
			=\big(m_j(y)\boldsymbol a^j(y)\big)\cdot
			\boldsymbol g_i(\varepsilon)(y,x_3).
			\]
			Hence, it follows that
			\begin{align*}
				m_i^\varepsilon(y,x_3)
				\boldsymbol g^i(\varepsilon)(y,x_3)
				&=\Big[
				\big(m_j(y)\boldsymbol a^j(y)\big)\cdot
				\boldsymbol g_i(\varepsilon)(y,x_3)
				\Big]\boldsymbol g^i(\varepsilon)(y,x_3)\\
				&=m_j(y)\boldsymbol a^j(y)
				=\boldsymbol m_S(y).
			\end{align*}
			Notice that
			\[
			g^{ij}(\varepsilon)(y,x_3)
			m_i^\varepsilon(y,x_3)m_j^\varepsilon(y,x_3)
			=|\boldsymbol m_S(y)|^2=1,
			\]
			and therefore
			$\boldsymbol m^\varepsilon\in\boldsymbol M(\varepsilon;\Omega)$.
			Since
			$\boldsymbol g_i(\varepsilon)\to\boldsymbol a_i$ in $\mathcal{C}^1(\overline{\Omega};\mathbb{E}^3)$, we deduce that, for each $i$,
			\[
			m_i^\varepsilon
			=\big(m_j\boldsymbol a^j\big)\cdot
			\boldsymbol g_i(\varepsilon)
			\longrightarrow
			\big(m_j\boldsymbol a^j\big)\cdot\boldsymbol a_i
			=m_i
			\quad\text{strongly in }H^1(\Omega),
			\]
			and consequently
			\[
			\boldsymbol m^\varepsilon\longrightarrow\boldsymbol m
			\quad\text{strongly in }(H^1(\Omega))^3.
			\]
			
			Since
			$m_p^\varepsilon=\boldsymbol m_S\cdot\boldsymbol g_p(\varepsilon)$
			and
			\begin{align*}
				\partial_\alpha\boldsymbol g_p(\varepsilon)
				&=\Gamma_{\alpha p}^{q}(\varepsilon)
				\boldsymbol g_q(\varepsilon),\\
				\frac1\varepsilon\partial_3\boldsymbol g_p(\varepsilon)
				&=\Gamma_{3p}^{q}(\varepsilon)
				\boldsymbol g_q(\varepsilon),
			\end{align*}
			we have, almost everywhere in $\Omega$,
			\begin{align*}
				m_{p\|\alpha}(\varepsilon;\boldsymbol m^\varepsilon)(y,x_3)
				&={}\partial_\alpha
				\big(\boldsymbol m_S(y)\cdot
				\boldsymbol g_p(\varepsilon)(y,x_3)\big)
				-\Gamma_{\alpha p}^{q}(\varepsilon)(y,x_3)
				\big(\boldsymbol m_S(y)\cdot
				\boldsymbol g_q(\varepsilon)(y,x_3)\big)\\
				&={}\partial_\alpha\boldsymbol m_S(y)\cdot
				\boldsymbol g_p(\varepsilon)(y,x_3),
			\end{align*}
			and
			\begin{align*}
				m_{p\|3}(\varepsilon;\boldsymbol m^\varepsilon)(y,x_3)
				&={}\frac1\varepsilon\partial_3
				\big(\boldsymbol m_S(y)\cdot
				\boldsymbol g_p(\varepsilon)(y,x_3)\big)
				-\Gamma_{3p}^{q}(\varepsilon)(y,x_3)
				\big(\boldsymbol m_S(y)\cdot
				\boldsymbol g_q(\varepsilon)(y,x_3)\big)\\
				&={}0.
			\end{align*}
			Therefore
			\begin{align*}
				&g^{pq}(\varepsilon)(y,x_3)
				m_{p\|\alpha}(\varepsilon;\boldsymbol m^\varepsilon)(y,x_3)
				m_{q\|\beta}(\varepsilon;\boldsymbol m^\varepsilon)(y,x_3)\\
				&\qquad={}
				\partial_\alpha\boldsymbol m_S(y)\cdot
				\partial_\beta\boldsymbol m_S(y).
			\end{align*}
			Since $g^{\alpha3}(\varepsilon)=0$, the exchange energy is consequently
			\[
			\frac12\int_\Omega
			g^{\alpha\beta}(\varepsilon)(y,x_3)
			\partial_\alpha\boldsymbol m_S(y)
			\cdot
			\partial_\beta\boldsymbol m_S(y)
			\sqrt{g(\varepsilon)(y,x_3)}\,\mathrm{d} x.
			\]
			The uniform convergence of the geometrical coefficients and integration
			over $(-1,1)$ show that this expression converges to
			\[
			\int_\omega
			|\nabla_S\boldsymbol m_S|^2\sqrt a\,\mathrm{d} y.
			\]
			Thus equality is attained in \eqref{sec5-eqn7}. Finally, the pullback
			of the associated magnetization is exactly
			$\boldsymbol m_S(y)$ and is independent of $x_3$.
			Therefore the result proved in part (ii) applies directly and yields
			\eqref{sec5-eqn8}, completing the proof.
		\end{proof}
		\begin{lem}[Lower bound]
			\label{sec5-lem3}
			Let
			$(\boldsymbol v(\varepsilon),\boldsymbol m(\varepsilon))$
			converge weakly in $\boldsymbol{\mathcal X}$ to
			$(\boldsymbol v,\boldsymbol m)$. Then the lower bound inequality
			\eqref{sec5-eqn4} holds.
		\end{lem}
		
		\begin{proof}
			We may assume, after extracting a subsequence, that the right-hand side
			of \eqref{sec5-eqn4} is finite and is realized as a limit. Lemma
			\ref{sec5-lem1} then shows that
			$\boldsymbol v=\boldsymbol\eta$ and
			$\boldsymbol m=\boldsymbol m(y)$ are independent of $x_3$, with
			$(\boldsymbol\eta,\boldsymbol m)
			\in\boldsymbol{\mathcal A}_M(\omega)$.
			
			The strong $L^2$ convergence and the uniform bounds for the
			magnetization give
			\begin{equation*}
				e_{\alpha\beta}
				(\varepsilon;\boldsymbol m(\varepsilon))
				\to e_{\alpha\beta}(\boldsymbol m)
				\quad\text{strongly in }L^2(\Omega).
			\end{equation*}
			On the other hand, the convergences of the Christoffel symbols in Lemma
			\ref{sec4:lem1} imply
			\[
			\epsilon_{\alpha\|\beta}
			(\varepsilon;\boldsymbol v(\varepsilon))
			\rightharpoonup
			\gamma_{\alpha\beta}(\boldsymbol\eta)
			\quad\text{weakly in }L^2(\Omega).
			\]
			Consequently,
			\begin{equation}
				\label{sec5-eqn10}
				E_{\alpha\|\beta}
				(\varepsilon;\boldsymbol v(\varepsilon),
				\boldsymbol m(\varepsilon))
				\rightharpoonup
				\mathcal E_{\alpha\beta}(\boldsymbol\eta,\boldsymbol m)
				\quad\text{weakly in }L^2(\Omega).
			\end{equation}
			
			For a symmetric matrix $(T_{ij})$, minimizing the three-dimensional
			quadratic form with respect to $T_{\alpha3}$ and $T_{33}$ gives
			\begin{equation}
				\label{sec5-eqn11}
				\frac12A^{ijk\ell}(\varepsilon)T_{k\ell}T_{ij}
				\geq
				\frac14a^{\alpha\beta\sigma\tau}(\varepsilon)
				T_{\sigma\tau}T_{\alpha\beta},
			\end{equation}
			where
			\begin{align*}
				a^{\alpha\beta\sigma\tau}(\varepsilon)
				&:=
				\frac{4\lambda\mu}{\lambda+2\mu}
				g^{\alpha\beta}(\varepsilon)
				g^{\sigma\tau}(\varepsilon)\\
				&\quad+
				2\mu\big(
				g^{\alpha\sigma}(\varepsilon)g^{\beta\tau}(\varepsilon)
				+g^{\alpha\tau}(\varepsilon)g^{\beta\sigma}(\varepsilon)
				\big).
			\end{align*}
			Equality in \eqref{sec5-eqn11} holds precisely when
			\begin{equation}
				\label{sec5-eqn12}
				T_{\alpha3}=0,
				\qquad
				T_{33}
				=
				-\frac{\lambda}{\lambda+2\mu}
				g^{\alpha\beta}(\varepsilon)T_{\alpha\beta}.
			\end{equation}
			Using \eqref{sec5-eqn10}, the uniform convergence of the coefficients,
			and weak lower semicontinuity, we obtain
			\begin{align*}
				&\liminf_{\varepsilon\to0}
				\frac12\int_\Omega
				A^{ijk\ell}(\varepsilon)
				E_{k\|\ell}(\varepsilon)E_{i\|j}(\varepsilon)
				\sqrt{g(\varepsilon)}\,\mathrm{d} x\\
				&\hspace{15mm}\geq
				\frac12\int_\omega
				a^{\alpha\beta\sigma\tau}
				\mathcal E_{\sigma\tau}(\boldsymbol\eta,\boldsymbol m)
				\mathcal E_{\alpha\beta}(\boldsymbol\eta,\boldsymbol m)
				\sqrt a\,\mathrm{d} y.
			\end{align*}
			
			Lemma \ref{sec5-lem2} supplies the lower bound for the exchange energy
			and the limit of the magnetostatic energy. Finally,
			\begin{align*}
				\int_\Omega f^iv_i(\varepsilon)
				\sqrt{g(\varepsilon)}\,\mathrm{d} x
				&\longrightarrow
				\int_\omega p^i\eta_i\sqrt a\,\mathrm{d} y,\\
				\int_\Omega h^im_i(\varepsilon)
				\sqrt{g(\varepsilon)}\,\mathrm{d} x
				&\longrightarrow
				\int_\omega q^im_i\sqrt a\,\mathrm{d} y.
			\end{align*}
			The first convergence follows from the weak convergence of the
			displacement, while the second follows from the strong $L^2$
			convergence of the magnetization. Adding all these inequalities proves
			\eqref{sec5-eqn4}.
		\end{proof}
		
		\begin{lem}[Upper bound]
			\label{sec5-lem4}
			For every
			$(\boldsymbol v,\boldsymbol m)\in\boldsymbol{\mathcal X}$ there exists
			a sequence converging strongly in $\boldsymbol{\mathcal X}$ for which
			the upper-bound inequality \eqref{sec5-eqn5} holds.
		\end{lem}
		
		\begin{proof}
			It is enough to consider a pair
			$(\boldsymbol\eta,\boldsymbol m)
			\in\boldsymbol{\mathcal A}_M(\omega)$. By density, choose
			$\boldsymbol\eta^n=(\eta_i^n)\in(\mathcal D(\omega))^3$ such that
			\begin{equation*}
				\eta_\alpha^n\to\eta_\alpha
				\quad\text{strongly in }H^1(\omega),
				\qquad
				\eta_3^n\to\eta_3
				\quad\text{strongly in }L^2(\omega).
			\end{equation*}
			Define
			\begin{equation*}
				F_{\alpha\beta}^n
				:=
				\gamma_{\alpha\beta}(\boldsymbol\eta^n)
				-m_\alpha m_\beta+\frac13a_{\alpha\beta}.
			\end{equation*}
			Then
			\begin{equation*}
				F_{\alpha\beta}^n
				\to
				\mathcal E_{\alpha\beta}(\boldsymbol\eta,\boldsymbol m)
				\quad\text{strongly in }L^2(\omega).
			\end{equation*}
			
			The transverse scaled derivatives that realize the plane-stress
			minimum corresponding to $F_{\alpha\beta}^n$ are
			\begin{align}
				\label{sec5-eqn13}
				z_\alpha^n
				&:=
				2m_\alpha m_3-\partial_\alpha\eta_3^n
				-2b_\alpha^\sigma\eta_\sigma^n,\\
				z_3^n
				&:=
				m_3^2-\frac13
				-\frac{\lambda}{\lambda+2\mu}
				a^{\alpha\beta}F_{\alpha\beta}^n.
				\nonumber
			\end{align}
			In general, $\boldsymbol z^n$ does not have zero trace
			\revised{on $\partial\omega$}. We therefore choose
			$\boldsymbol z^{n,k}\in(\mathcal D(\omega))^3$ such that
			\begin{equation}
				\label{sec5-eqn14}
				\boldsymbol z^{n,k}\to\boldsymbol z^n
				\quad\text{strongly in }(L^2(\omega))^3
				\quad\text{as }k\to\infty.
			\end{equation}
			
			For fixed $n$ and $k$, define
			\begin{equation*}
				v_i^{\varepsilon,n,k}(y,x_3)
				:=
				\eta_i^n(y)+\varepsilon x_3z_i^{n,k}(y).
			\end{equation*}
			Since both $\boldsymbol\eta^n$ and $\boldsymbol z^{n,k}$ have compact
			support in $\omega$, one has
			$\boldsymbol v^{\varepsilon,n,k}\in\boldsymbol V(\Omega)$. For the
			magnetization, use the extension
			\begin{equation*}
				m_i^\varepsilon(y,x_3)
				:=
				\boldsymbol m_S(y)\cdot
				\boldsymbol g_i(\varepsilon)(y,x_3).
			\end{equation*}
			By Lemma \ref{sec5-lem2},
			$\boldsymbol m^\varepsilon
			\in\boldsymbol M(\varepsilon;\Omega)$ and
			$\boldsymbol m^\varepsilon\to\boldsymbol m$ strongly in
			$(H^1(\Omega))^3$.
			
			Lemma \ref{sec4:lem1} and the definitions of the scaled strains imply,
			for fixed $n$ and $k$,
			\begin{align}
				\label{sec5-eqn15}
				E_{\alpha\|\beta}
				(\varepsilon;\boldsymbol v^{\varepsilon,n,k},
				\boldsymbol m^\varepsilon)
				&\to F_{\alpha\beta}^n,\nonumber\\
				E_{\alpha\|3}
				(\varepsilon;\boldsymbol v^{\varepsilon,n,k},
				\boldsymbol m^\varepsilon)
				&\to
				\frac12\big(z_\alpha^{n,k}+\partial_\alpha\eta_3^n\big)
				+b_\alpha^\sigma\eta_\sigma^n-m_\alpha m_3,\\
				E_{3\|3}
				(\varepsilon;\boldsymbol v^{\varepsilon,n,k},
				\boldsymbol m^\varepsilon)
				&\to z_3^{n,k}-m_3^2+\frac13
				\nonumber
			\end{align}
			strongly in $L^2(\Omega)$. Letting first $\varepsilon\to0$ and then
			$k\to\infty$, relations \eqref{sec5-eqn13}--\eqref{sec5-eqn14} show that
			the last two limits in \eqref{sec5-eqn15} become
			\[
			0
			\qquad\text{and}\qquad
			-\frac{\lambda}{\lambda+2\mu}
			a^{\alpha\beta}F_{\alpha\beta}^n,
			\]
			respectively. These are exactly the equality conditions in
			\eqref{sec5-eqn12}. Hence
			\begin{align}
				\label{sec5-eqn16}
				&\lim_{k\to\infty}\lim_{\varepsilon\to0}
				\frac12\int_\Omega
				A^{ijk\ell}(\varepsilon)
				E_{k\|\ell}
				(\varepsilon;\boldsymbol v^{\varepsilon,n,k},
				\boldsymbol m^\varepsilon)
				E_{i\|j}
				(\varepsilon;\boldsymbol v^{\varepsilon,n,k},
				\boldsymbol m^\varepsilon)
				\sqrt{g(\varepsilon)}\,\mathrm{d} x\nonumber\\
				&\hspace{20mm}=
				\frac12\int_\omega
				a^{\alpha\beta\sigma\tau}
				F_{\sigma\tau}^nF_{\alpha\beta}^n
				\sqrt a\,\mathrm{d} y.
			\end{align}
			
			The same magnetic recovery sequence satisfies
			\begin{align}
				\label{sec5-eqn17}
				&\lim_{\varepsilon\to0}
				\frac12\int_\Omega
				g^{ij}(\varepsilon)g^{pq}(\varepsilon)
				m_{p\|i}(\varepsilon;\boldsymbol m^\varepsilon)
				m_{q\|j}(\varepsilon;\boldsymbol m^\varepsilon)
				\sqrt{g(\varepsilon)}\,\mathrm{d} x \nonumber\\
				&\hspace{20mm}=
				\int_\omega|\nabla_S\boldsymbol m_S|^2\sqrt a\,\mathrm{d} y,
			\end{align}
			and
			\begin{equation}
				\label{sec5-eqn18}
				\lim_{\varepsilon\to0}
				\frac{\mu_0}{2\varepsilon}
				\int_{\mathbb R^3}
				|\nabla v_{\varepsilon,\boldsymbol m^\varepsilon}|^2\,\mathrm{d} \xi
				=
				\frac1{\mu_0}\int_\omega m_3^2\sqrt a\,\mathrm{d} y.
			\end{equation}
			Finally,
			\begin{align}
				\label{sec5-eqn19}
				\lim_{\varepsilon\to0}
				\int_\Omega f^iv_i^{\varepsilon,n,k}
				\sqrt{g(\varepsilon)}\,\mathrm{d} x
				&=
				\int_\omega p^i\eta_i^n\sqrt a\,\mathrm{d} y,\nonumber\\
				\lim_{\varepsilon\to0}
				\int_\Omega h^im_i^\varepsilon
				\sqrt{g(\varepsilon)}\,\mathrm{d} x
				&=
				\int_\omega q^im_i\sqrt a\,\mathrm{d} y.
			\end{align}
			
			Combining \eqref{sec5-eqn16}--\eqref{sec5-eqn19}, then letting
			$k\to\infty$ and $n\to\infty$, gives
			\[
			\lim_{n\to\infty}\lim_{k\to\infty}
			\lim_{\varepsilon\to0}
			\mathcal G(\varepsilon;
			\boldsymbol v^{\varepsilon,n,k},\boldsymbol m^\varepsilon)
			=
			\mathcal G_M(\boldsymbol\eta,\boldsymbol m).
			\]
			A standard diagonal argument yields choices
			$n=n(\varepsilon)\to\infty$ and
			$k=k(\varepsilon)\to\infty$ \revised{for which the construction above gives} an admissible sequence converging
			strongly in $\boldsymbol{\mathcal X}$ and satisfying
			\eqref{sec5-eqn5}.
		\end{proof}
		
		\begin{proof}[Proof of Theorem \ref{sec5-thm1}]
			Lemma \ref{sec5-lem1} proves the compactness results. Lemma \ref{sec5-lem3} is the $\Gamma$-liminf inequality, while Lemma \ref{sec5-lem4} supplies a
			recovery sequence and therefore the $\Gamma$-limsup inequality. Our proof is complete.
		\end{proof}
		
		To conclude the paper, we have the following, which is a direct consequence of Theorem \ref{sec5-thm1}.
		
		\begin{cor}[Convergence of minima and minimizers]
			Problem \hyperref[membrane]{$\mathcal P_M(\omega)$} admits at least one
			solution and
			\begin{equation*}
				\lim_{\varepsilon\to0}
				\min_{\boldsymbol{\mathcal A}(\varepsilon;\Omega)}
				\mathcal G(\varepsilon;\cdot,\cdot)
				=
				\min_{\boldsymbol{\mathcal A}_M(\omega)}
				\mathcal G_M.
			\end{equation*}
			Let
			$(\boldsymbol u(\varepsilon),\boldsymbol m(\varepsilon))$ be a family
			of solutions of Problem
			\hyperref[problem-scaled]{$\mathcal P(\varepsilon;\Omega)$}.
			\begin{revisedpassage}
				For $i=1,2,3$, let $\overline u_i(\varepsilon)$ denote the thickness average of the scaled covariant displacement component:
				$$
				\overline u_i(\varepsilon)(y)
				:=\frac12\int_{-1}^{1}u_i(\varepsilon)(y,x_3)\,\mathrm{d} x_3
				\qquad\text{for a.e. }y\in\omega.
				$$
			\end{revisedpassage}
			Then every sequence $\varepsilon\to0$ has a subsequence, not relabeled, and a
			solution $(\boldsymbol\zeta,\boldsymbol m)$ of Problem
			\hyperref[membrane]{$\mathcal P_M(\omega)$} such that
			\begin{align*}
				\overline u_\alpha(\varepsilon)&\to\zeta_\alpha
				&&\text{strongly in }H^1(\omega),\\
				\overline u_3(\varepsilon)&\to\zeta_3
				&&\text{strongly in }L^2(\omega),\\
				u_i(\varepsilon)&\to\zeta_i
				&&\text{strongly in }L^2(\Omega),\\
				m_i(\varepsilon)&\to m_i
				&&\text{strongly in }H^1(\Omega).
			\end{align*}
			\revised{In the last two convergence statements, the functions $\zeta_i$ and $m_i$ are identified with their $x_3$-independent extensions to $\Omega$.}
		\end{cor}

		\section*{Declarations}
		The author has no conflicts of interest to declare that are relevant to the content of this article.
		
		\section*{Data availability}
		No data were generated or analysed in this study.

		\section*{Acknowledgements} 
		The author has been partially supported by the grant PRIMUS/24/SCI/020 of Charles University and by the project Ferroic Multifunctionalities (FerrMion) [reg. no. CZ.02.01.01/00/22 008/0004591].

		\newcommand \auth{} 
		\newcommand \jour {} 
		\newcommand \book {}

	\end{document}